\documentclass[12pt, reqno]{amsart}
\usepackage{amsfonts}
\usepackage{amssymb}

\theoremstyle{plain}

\newtheorem{theo}[equation]{Theorem}
\newtheorem{pro}[equation]{Proposition}
\newtheorem{coro}[equation]{Corollary}
\newtheorem{lem}[equation]{Lemma}
\newtheorem{defi}[equation]{Definition}
\newtheorem{rem}[equation]{Remark}
\def\vn{\varepsilon}
\def\ot{\otimes}
\def\om{\omega}
\def\mg{\mathfrak{g}}
\def\q{{\mathbf{q}}}

\def\lan{\langle}
\def\ran{\rangle}
\def\al{\alpha}

\def\be{\beta}
\def\De{\Delta}
\def\ga{\gamma}
\def\Ga{\Gamma}
\def\Om{\Omega}
\def\si{\sigma}
\def\om{\omega}
\def\vep{\varepsilon}

\def\de{\delta}
\def\pa{\partial}
\def\La{\Lambda}
\def\la{\lambda}
\def\bi{\binom}

\def\F{{\mathcal F}}

\def\O{{\mathcal O}}

\def\bN{{\mathbb N}}
\def\bZ{{\mathbb Z}}
\def\bQ{{\mathbb Q}}

\def\bK{{\mathbb K}}
\def\bC{{\mathbb C}}
\def\lra{\longrightarrow}

\def\dim{\mbox{\rm dim}\,}

\def\q{{\bf{q}}}

\def\lg{\langle}
\def\rg{\rangle}

\def\pa{\partial}
\def\lra{\longrightarrow}
\def\al{\alpha}
\def\be{\beta}
\def\ga{\gamma}
\def\si{\sigma}

\def\De{\Delta}
\def\de{\delta}

\def\vn{\varepsilon}

\def\ot{\otimes}
\def\la{\lambda}
\def\La{\Lambda}
\def\om{\omega}
\def\lra{\longrightarrow}
\def\mg{\mathfrak {g}}

\def\sm1{_{(-1)}}
\def\s1{_{(1)}}
\def\s2{_{(2)}}
\def\s3{_{(3)}}
\def\s0{_{(0)}}
\def\um1{^{(-1)}}
\def\u1{^{(1)}}
\def\u2{^{(2)}}
\def\u3{^{(3)}}
\def\u0{^{(0)}}
\def\id{\text{\rm id}}

\def\id{\text{\rm id}}

\def\bi{\binom}
\def\wt{\mbox{\rm wt}}

\title[Multi-parameter quantum groups ]{Multi-parameter quantum groups \\ and quantum shuffles, (I)}
\author[Hu]{Naihong Hu$^\star$}
\address{Department of Mathematics, East China Normal University,
Minhang Campus, Dong Chuan Road 500, Shanghai 200241, PR China}
\email{nhhu@math.ecnu.edu.cn}
\thanks{$^\star$N.H.,
supported in part by the NNSFC (Grant 10728102), the PCSIRT and the
SRFDP from the MOE, the National/Shanghai Priority Academic
Discipline Programmes (Project Number: B407).}

\author[Pei]{Yufeng Pei$^\dag$}
\address{Department of Mathematics, Shanghai Normal University, Guilin Road 100,
Shanghai 200234 PR China}\email{peiyufeng@gmail.com}
\thanks{$^\dag$Y.Pei,
corresponding author, supported in part by the NNSFC (Grant
10571119) and the ZJNSF (Grant Y607136).}

\author[Rosso]{Marc Rosso}
\address{UFR de Math\'ematiques,
Universit\'e Denis Diderot - Paris 7, 2 Place Jussieu, 75251 Paris
Cedex 05, France} \email{rosso@math.jussieu.fr}

\date\today
\begin{document}

\begin{abstract}
In this article, we study the multi-parameter quantum groups defined
by generators and relations associated with symmetrizable
generalized Cartan matrices, together with their representations in
the category $\mathcal O$. This presentation will be convenient for
our later discussions. We present two explicit descriptions here: as
a Hopf $2$-cocycle deformation, and as the multi-parameter quantum
shuffle realization of the positive part.
\end{abstract}
\maketitle

\section{Introduction}
In the early 90s last century, much work has been done on the
multi-parameter deformations of the coordinate algebra of the
general linear algebraic group. These deformations were firstly
described in \cite{AST} and independently in \cite{Re}. These
implied that multi-parameter deformations can be obtained by
twisting the coalgebra structure \cite{Re} in the spirit of Drinfeld
[9] or by twisting the algebra structure via a $2$-cocycle on a free
abelian group \cite{AST}. In fact, the original work of Drinfeld and
Reshetikhin concerned only with quasitriangular Hopf algebras, but
their constructions can be dualised to the case of
co-quasitriangular Hopf algebras by Hopf $2$-cocycle deformations
\cite{DT,Ma}.

Benkart-Witherspoon \cite{BW1, BW2} investigated a class of
two-parameter quantum groups $U(\mathfrak{gl}_n)$ and
$U(\mathfrak{sl}_n)$ of type $A$ arising from the work on down-up
algebras \cite{BW0}, which were early defined by Takeuchi \cite{Ta}.
Bergeron-Gao-Hu \cite{BGH1, BGH2} developed the corresponding theory
for two-parameter quantum orthogonal and symplectic groups, in
particular, they studied the distinguished Lusztig's symmetries
property for the two-parameter quantum groups of classical type.
Recently, this fact has been generalized to the cases of Drinfeld
doubles of bosonizations of Nichols algebras of diagonal type by
Heckenberger in \cite{He}, that is, the study of Lusztig
isomorphisms (only existed among a family of different objects) in
the multi-parameter setting finds a beautiful realization model for
his important notion of Weyl groupoid defined in \cite{He2}. It
should be pointed out that this is also a remarkable feature for the
multi-parameter quantum groups in question that are distinct from
the one-parameter ones familiar to us (see \cite{Lu}).

Hu-Shi \cite{HS}, Bai-Hu \cite{BH} did contributions to exceptional
types $G_2$, $E$, respectively; Hu-Wang \cite{HW1, HW2}, Bai-Hu and
Chen-Hu-Wang further investigated the structure theory of
two-parameter restricted quantum groups for types $B$, $G_2$, $D$
and $C$ at roots of unity, including giving the explicit
constructions of convex PBW-type Lyndon bases with detailed
information on commutation relations, determining the isomorphisms
as Hopf algebras and integrals, as well as necessary and sufficient
conditions for them to be ribbon Hopf algebras.

Another new interesting development is the work of Hu-Rosso-Zhang
and Hu-Zhang \cite{HRZ, HZ1, HZ2} achieved for affine types
$X^{(1)}_\ell$, where $X=A, B, C, D, E, F_4, G_2$. Of most
importance among them are the following: (1) Drinfeld realizations
in the two-parameter setting were worked out; (2) Axiomatic
definition for Drinfeld realizations was achieved in terms of
inventing $\tau$-invariant generating function; (3) Quantum affine
Lyndon bases were put forwarded and constructed for the first time;
(4) Constructions of two-parameter vertex representations of level
$1$ for $X^{(1)}_\ell$ were obtained.

Using the Euler form, the first two authors \cite{HP1} introduced a
unified definition for a class of two-parameter quantum groups for
all types  and studied their structure. Shortly after, this
definition was quoted in \cite{BS}. On the other hand,
(multi)two-parameter quantum groups have been deeply related to many
interesting work. For instance, Krob and Thibon \cite{KT1} on
noncommutative symmetric functions; Reineke \cite{R1} on generic
extensions and degenerate two-parameter quantum groups of
simply-laced cases, and the classifications of Artin-Shelt regular
algebras \cite{LPWZ}.

In \cite{Ro0, Ro1}, the third author found a realization of
$U_{q}^+$, the positive part of the standard quantized enveloping
algebra associated with a Cartan matrix by quantizing the shuffle
algebra (see also \cite{FG1, Gr2, Le}). It was mentioned that the
supersymmetric and multi-parameter versions of $U_{q}^+$ (for a
suitable choice of the Hopf bimodule) also can be treated in this
uniform principle. From a more recent point of view, Andruskiewich
and Schneider obtained remarkable results on the structure of
pointed Hopf algebras arising from Nichols algebras (or say, quantum
symmetric algebras as in \cite{Ro1}) and their lifting method
\cite{AS1, AS2, AS3}.

In this paper, we study a class of multi-parameter quantum groups
$U_{\q}(\mg_A)$ defined by generators and relations associated with
symmetrizable generalized Cartan matrices $A$, together with their
representations in the category $\mathcal O$. In section 2, we show
that $U_{\q}(\mg)$ can be realized as Drinfeld doubles of certain
Hopf subalgebras with respect to a Hopf skew-pairing
$\lan\,,\,\ran_{\q}$, and as a consequence, it has a natural
triangular decomposition. Partially motivated by Doi-Takeuchi
\cite{DT}, Majid \cite{Ma} and also Westreich \cite{W} on Hopf
$2$-cocycle deformation theory, we construct an explicit Hopf
$2$-cocycle on $U_{q,q^{-1}}(\mg_A)$ and use it to twist its
multiplication to get the required multi-parameter quantum group
$U_{\q}(\mg_A)$. In section 3, the representation theory of
$U_{\q}(\mg_A)$ under the assumption that $q_{ii} \; (i\in I)$ are
not roots of unity is described, which is the generalization of the
corresponding one for two-parameter quantum groups of types $A, B,
C, D$ developed in \cite{BGH2} and \cite{BW2}. We show that the Hopf
skew-pairing $\lan\,,\,\ran_{\q}$ is non-degenerate when restricted
to each grading component. In section 4, using a non-degenerate
$\tau$-sesquilinear form on $U_{\q}^+$ (where $\tau$ is an
involution automorphism of the ground field
$\bK\supset\bQ(q_{ij\;|\;i,j\in\ I})$ such that
$\tau(q_{ij})=q_{ji}, \ i,j\in I$), we prove that the positive part
$U_{\q}^+$ of $U_{\q}(\mg_A)$ can be embedded into the
multi-parameter quantum shuffle algebra $(\F,\star)$. It turns out
that this realization plays a key role both in the study of
PBW-bases of $U_{\q}(\mg_A)$ and the construction of multi-parameter
Ringel-Hall algebras (see \cite{P1} for more details).

Throughout the paper, we  denote by $\bZ,\, \bZ_+,\, \bN$, $\bC$ and
$\bQ$ the set of integers, the set of non-negative integers, the set
of positive integers, the set of complex numbers and the set of
rational numbers, respectively.

\section{Multi-parameter quantum group and Hopf $2$-cocycle deformation}

\subsection{}
Let us start with some notations. For $n>0$, define
$$
(n)_v=\frac{v^n-1}{v-1}.
$$
$$
(n)_v!=(n)_v\cdots(2)_v(1)_v, \quad \textit{and}\quad (0)_v!=1.
$$
$$
\binom{n}{k}_v=\frac{(n)_v!}{(k)_v!(n-k)_v!}.
$$

The following identities are well-known.
\begin{eqnarray}
&&(m+n)_{v}=(m)_{v}+v^{m}(n)_{v},\\
&&\bi{m}{k}_{v}(m{-}k)_{v}=\bi{m}{k{+}1}_{v}(k{+}1)_{v},\\
&&\bi{r}{k}_{v}\bi{k}{m}_{v}\bi{r{-}k}{n}_{v}=\bi{r{-}m{-}n}{k{-}m}_{v}\bi{m{+}n}{m}_{v}\bi{r}{m{+}n}_{v},\\
&&\bi{n}{k}_v=v^k\bi{n{-}1}{k}_v+\bi{n{-}1}{k{-}1}_v=\bi{n{-}1}{k}_v+v^{n-k}\bi{n{-}1}{k{-}1}_v,\\
&&\sum_{k=0}^{n}(-1)^k\bi{n}{k}_{v}v^{\frac{k(k-1)}{2}}a^{n-k}z^k=\prod_{k=0}^{n-1}(a-vz^k),\
\ \forall\ \ \textit{scalar}\  a.
\end{eqnarray}

\subsection{}
Assume that $R$ is a field ($\text{char} R\neq2$) with an
automorphism $\tau$. Let V be a $R$-vector space. A {\it
$\tau$-linear map} $f$ on $V$ is a function: $V\to R$ such that
$$
f(av)=\tau(a)f(v),\qquad\text{for any }\ a\in R,\ v\in V.
$$
A {\it$\tau$-sesquilinear form} $f$ on $V$ is a function: $V\times
V\to R$, subject to the conditions:
\begin{eqnarray*}
&&f(x+y,z)=f(x,z)+f(y,z),\\
&&f(x,y+z)=f(x,y)+f(x,z),\\
&&f(ax,y)=\tau(a)f(x,y)=f(x,\tau(a)y),\qquad\forall\ a\in R,
\end{eqnarray*}
for any $x,y,z\in V$. If $\tau$ is the identity, $f$ is an ordinary
bilinear form on $V$. A $\tau$-sesquilinear form $f$ with
$\tau^2=\id$ is called {\it $\tau$-Hermitian form} if $\tau
(f(x,y))=f(y,x)$ for any $x,y\in V$. If $\tau=\id$, $f$ is a
symmetric bilinear form on $V$.

\subsection{}
Let $\mg_A$ be a symmetrizable Kac-Moody algebra over $\bQ$  and
$A=(a_{ij})_{i,j\in I}$ be an associated  generalized Cartan matrix.
Let $d_i$ be relatively prime positive integers such that
$d_ia_{ij}=d_{j}a_{ji}$ for $i,j\in I$. Let $\Phi$ be the root
system, $\Pi=\{\al_i\mid i\in I\}$ a set of simple roots,
$Q=\bigoplus_{i\in I}\bZ\al_i$ the root lattice, and then with
respect to $\Pi$, we have $\Phi^+$ the system of positive roots,
$Q^+=\bigoplus_{i\in I}\bZ_+\al_i$ the positive root lattice, $\La$
the weight lattice, and $\La^+$ the set of dominant weights. Let
$q_{ij}$ be indeterminates over $\bQ$ and $\bQ(q_{ij}\;|\;i,j\in I)$
be the fraction field of polynomial ring $\bQ[q_{ij}\;|\;i,j\in I]$
such that
\begin{eqnarray}
q_{ij}q_{ji}=q_{ii}^{a_{ij}}.
\end{eqnarray}
Let $\bK\supset\bQ(q_{ij\;|\;i,j\in\ I})$ be a field such that
$q_{ii}^{\frac{1}{m}}\in\bK$ for $m\in\bZ_{+}$. Assume that there
exists an involution $\bQ$-automorphism $\tau$ of $\bK$ such that
$\tau(q_{ij})=q_{ji}$. Denote $\q:=(q_{ij})_{i,j\in I}$.

\smallskip
\begin{defi}\label{defi}
The multi-parameter quantum group $U_{\q}(\mg_A)$ is an associative
algebra over $\bK$ with $1$ generated by the elements $e_i, f_i,
\om_i^{\pm1}, \om_i'^{\pm1} \ (i\in I)$, subject to the relations:
\begin{eqnarray*}
 &(R1)&\quad\om_i^{\pm1}\om_j'^{\pm1}=\om_j'^{\pm1}\om_i^{\pm1},\quad
\om_i^{\pm1}\om_i^{\mp1}=\om_i'^{\pm1}\om_i'^{\mp1}=1,\\
 &(R2)&\quad\om_i^{\pm1}\om_j^{\pm1}=\om_j^{\pm1}\om_i^{\pm1},
 \quad\om_i'^{\pm1}\om_j'^{\pm1}=\om_j'^{\pm1}\om_i'^{\pm1},\\
&(R3)&\quad \om_ie_j\om_i^{-1}=q_{ij} e_j,
\qquad\om_i'e_j\om_i'^{-1}=q_{ji}^{-1} e_j,
\\
&(R4)&\quad\om_if_j\om_i^{-1}=q_{ij}^{-1} f_j,\qquad
\om_i'f_j\om_i'^{-1}=q_{ji} f_j,\\
&(R5)&\quad[\,e_i, f_j\,]=\delta_{i,j}\frac{q_{ii}}{q_{ii}-1}({\om_i-\om_i'}),\\
&(R6)&\quad \sum_{k=0}^{1-a_{ij}} (-1)^{k}
\binom{1-a_{ij}}{k}_{q_{ii}}q_{ii}^{\frac{k(k-1)}{2}}
q_{ij}^{k}e_{i}^{1-a_{ij}-k} e_{j} e_{i}^{k} =0 \quad (i\ne j),\\
&(R7)&\quad\sum_{k=0}^{1-a_{ij}} (-1)^{k}
\binom{1-a_{ij}}{k}_{q_{ii}}q_{ii}^{\frac{k(k-1)}{2}}
q_{ij}^{k}f_{i}^{k} f_{j} f_{i}^{1-a_{ij}-k}=0 \quad (i\ne j).
\end{eqnarray*}
\end{defi}

\begin{pro}\label{Hopf}
The associative algebra $U_{\q}(\mg)$ has a Hopf algebra structure
with the comultiplication, the counit and the antipode given by:
\begin{eqnarray*}
&& \Delta(\om_i^{\pm1})=\om_i^{\pm1}\ot\om_i^{\pm1}, \qquad
\Delta({\om_i'}^{\pm1})={\om_i'}^{\pm1}\ot{\om_i'}^{\pm1},\\
&&\Delta(e_i)=e_i\ot 1+\om_i\ot e_i, \qquad \Delta(f_i)=1\ot
f_i+f_i\ot \om_i',\\
&&\vn(\om_i^{\pm1})=\vn({\om_i'}^{\pm1})=1, \qquad\quad
\vn(e_i)=\vn(f_i)=0, \\
&&S(\om_i^{\pm1})=\om_i^{\mp1}, \qquad\qquad\qquad
S({\om_i'}^{\pm1})={\om_i'}^{\mp1},
\\
&&S(e_i)=-\om_i^{-1}e_i,\qquad\qquad\qquad
S(f_i)=-f_i\,{\om_i'}^{-1}.
\end{eqnarray*}
\end{pro}

\begin{rem}
$(1)$\ Assume that $q_{ij}=q^{d_ia_{ij}}\ (i,j\in I)$. In this case,
we denote $U_{q,q^{-1}}(\mg_A):=U_{\q}(\mg_A)$, and
$$
U_{\q}(\mg_A)/(\om_i'-\om_i^{-1})\simeq U_q(\mg_A),
$$
where $ U_q(\mg_A) $ is the one-parameter quantum group of
Drinfeld-Jimbo type \cite{Ja}.

$(2)$\  Assume that $q_{ij}=r^{\lan j,i\ran}s^{-\lan i,j\ran}$,
where
$$
\lan i,j\ran:=
\begin{cases}
&d_ia_{ij}\quad i<j,\\
&d_i\quad\quad\, i=j,\\
&0\quad\quad\ \   i>j.
\end{cases}
$$
$U_{\q}(\mg_A)$ is one of a class of two-parameter quantum groups
introduced uniformly by Hu-Pei \cite{HP1}, which, owing to
nonuniqueness of definitions for two-parameter quantum groups, have
some overlaps with the former examples defined in such as
\cite{BGH1, BGH2, BH, BKL, BW0, BW1, BW2, HP1, HS, HRZ} and
references therein.

$(3)$\  Assume that $\mg_A$ is of finite type and
$q_{ij}=q^{-u(\al_i,\al_j)-d_ia_{ij}}$, where $u$ is a skew
$\bZ$-bilinear form on root lattice $Q$. Then $U_{\q}(\mg_A)$ is the
multi-parameter quantum group $U_{q,Q}$ introduced by Hodge et al
\cite{Ho, HLT}. Note that the Hopf dual objects of these quantum
groups are isomorphic to those quantum groups discussed by
Reshetikhin \cite{Re} $($also see \cite{CV1}$)$.

$(4)$\  Assume that $q_{ij}=q^{d_ia_{ij}}p_{ij}$ where
$P=(p_{ij})_{i,j\in I}$ such that $p_{ij}p_{ji}=1, {p_{ii}}=1$. Then
$U_{\q}(\mg)$ are the multi-parameter quantum groups $U_{q,P}$
introduced by Hayashi in \cite{Ha}.
\smallskip

$(5)$\  Assume that $\mg=A_n$, $U_{\q}(\mg)$ is the multi-parameter
quantum groups or their dual object studied by many authors $($see
\cite{AE}, \cite{AST}, \cite{CM},  and references therein$)$.
\end{rem}

\begin{rem} The definition of $U_{\q}(\mg_A)$ has appeared in \cite{F1, F}.
The positive part of $U_{\q}(\mg_A)$ has appeared in \cite{Ro1}. The
Borel part of $U_{\q}(\mg_A)$ has appeared in \cite{Kh2}.
\end{rem}

{\bf From now on, we always assume that $q_{ii}$ are not roots of
unity.}

\subsection{}
Note that $\tau: \bK\to\bK$ that is defined by $\tau(q_{ij})=q_{ji}$
for $i,j\in I$ is a $\bQ$-automorphism of $\bK$.
\begin{lem}\label{auto}{\quad}\\
$(1)$\ There is a $\tau$-linear $\bQ$-algebra automorphism $\Phi$ of
$U_{\q}(\mg)$ defined by
\begin{eqnarray}
e_i\mapsto f_i,\quad f_i\mapsto e_i,\quad
\om_i\mapsto\om_i',\quad \om_i'\mapsto\om_i.
\end{eqnarray}

\noindent$(2)$\  There is a $\bK$-algebra anti-automorphism $\Psi$
of $U_{\q}(\mg)$ defined by
\begin{eqnarray}
e_i\mapsto f_i,\quad f_i\mapsto
e_i,\quad\om_i\mapsto\om_i,\quad\om_i'\mapsto\om_i'.
\end{eqnarray}
\end{lem}
\begin{proof}(2) is clear. (1) is due to the fact:
The $\q$-Serre relation
\begin{eqnarray*} &&\sum_{k=0}^{1-a_{ij}}
(-1)^{k} \binom{1-a_{ij}}{k}_{q_{ii}} q_{ii}^{\frac{k(k-1)}{2}}
q_{ij}^{k} e_{i}^{1-a_{ij}-k} e_{j} e_{i}^{k}=0,
\end{eqnarray*}
is equivalent to
\begin{eqnarray*}
&&\sum_{k=0}^{1-a_{ij}} (-1)^{k} \binom{1-a_{ij}}{k}_{q_{ii}}
q_{ii}^{\frac{k(k-1)}{2}}
 q_{ji}^{k}e_{i}^{k} e_{j} e_{i}^{1-a_{ij}-k}=0.
\end{eqnarray*}

This completes the proof.
\end{proof}

\subsection{}
It will be convenient to work with the algebra $\tilde{U}_{\q}(\mg)$
defined by the same generators
$e_i,f_i,\om_{i}^{\pm1},\om_{i}'^{\pm1}$ for $i\in I$, and subject
to relations $(R1)$---$(R5)$ only (without Serre relations). We have
the canonical homomorphism $\tilde{U}_{\q}(\mg)\twoheadrightarrow
{U}_{\q}(\mg)$. We abuse the notations both for the corresponding
elements in $\tilde{U}_{\q}(\mg)$ and ${U}_{\q}(\mg)$, which will be
clear from the context. For any $i,j\in I$ with $i\neq j$, set
\begin{eqnarray}
&&u_{ij}^+:=\sum_{k=0}^{1-a_{ij}} (-1)^{k} \binom{1-a_{ij}}{k}_{q_{ii}}
q_{ii}^{\frac{k(k-1)}{2}} q_{ij}^{k} e_{i}^{1-a_{ij}-k} e_{j}
e_{i}^{k},\\
&&u_{ij}^-:=\sum_{k=0}^{1-a_{ij}} (-1)^{k}
\binom{1-a_{ij}}{k}_{q_{ii}}q_{ii}^{\frac{k(k-1)}{2}}
q_{ij}^{k}f_{i}^{k} f_{j} f_{i}^{1-a_{ij}-k}.
\end{eqnarray}

\begin{lem}\label{primi}
Let $i,j\in I$ with $i\neq j$. Then
\begin{eqnarray*}
&&\De(u_{ij}^+)=u_{ij}^+\ot 1 +\om_i^{1-a_{ij}}\om_{j}\ot u_{ij}^+,
\quad\De(u_{ij}^-)=u_{ij}^-\ot \om_i'^{1-a_{ij}}\om_{j}' +1\ot u_{ij}^-.
\end{eqnarray*}
\end{lem}
\begin{proof}
See Appendix A.
\end{proof}

\subsection{}
Let $U_{\q}^+$ (respectively, $U_{\q}^-$) be the subalgebra of
$U_{\q}$ generated by the elements $e_i$ (respectively, $f_i$) for
$i\in I$,  $U_{\q}^{+0}$ (respectively, $U_{\q}^{-0}$) the
subalgebra of $U_{\q}$ generated by $\om_i^{\pm1}$ (respectively,
$\om_i'^{\pm1}$)  for $i\in I$. Let $U_{\q}^{0}$ be the subalgebra
of $U_{\q}$ generated by $\om_i^{\pm1},\om_i'^{\pm1}$  for $i\in I$.
Moreover, Let $U_{\q}^{\leq0}$ (respectively, $U_{\q}^{\geq0}$) be
the subalgebra of $U_{\q}$ generated by the elements
$e_i,\om_i^{\pm1}$ for $i\in I$ (respectively, $f_i, \om_i'^{\pm1}$
for $i\in I$). It is clear that $U_{\q}^0, {U}_{\q}^{\pm0}$ are
commutative algebras.  Similarly, we can define $\tilde{U}_{\q}^+,
\tilde{U}_{\q}^-,\tilde{U}_{\q}^0$, etc. For each $\mu\in Q$, we can
define the elements $\om_{\mu}$ and $\om_{\mu}'$ by
$$
\om_{\mu}=\prod_{i\in I}\om_{i}^{\mu_i},\qquad
\om'_{\mu}=\prod_{i\in I}{\om'_{i}}^{\mu_i}
$$
if $\mu=\sum_{i\in I}\mu_i\al_i\in Q$. For any $\mu,\ \nu\in Q$, we
denote
$$
q_{\mu\nu}:=\prod_{i,j\in I} q_{ij}^{\mu_i\nu_j}
$$
if $\mu=\sum_{i\in I}\mu_i \al_i$ and $\nu=\sum_{j\in I}\nu_j \al_j$.
Let
$$
\deg e_i=\al_i,\quad \deg f_i=-\al_i,\quad \deg \om_i^{\pm1}=\deg
\om_i'^{\pm1}=0.
$$
Then
$$
U_{\q}^{\pm}=\bigoplus_{\be\in Q^+}(U_{\q}^{\pm})_{\pm \be},
$$
where
$$
(U_{\q}^{\pm})_{\pm \be}=\left\{x\in
U_{\q}^{\pm}\,\left|\,\om_{\mu}x\om_{-\mu} =q_{\mu\be}x,\
\om_{\mu}'x\om_{-\mu}' =q_{\be\mu}^{-1}x,\ \forall\ \mu\in
Q\right\}\right..
$$

\smallskip

\subsection{(Skew) Hopf pairings}
For $i\in I$, we define a linear form $\tau_i$ on $ U_{\q}^{\geq0}$
by
$$
\tau_i(e_i\om_{\mu})=\frac{q_{ii}}{1-q_{ii}},\qquad\text{for all}\
\mu\in Q,
$$
and
$$
\tau_i(U_{\nu}^{\geq0})=0,\qquad\text{for all}\ \nu\in Q\
\text{with}\ \nu\neq\al_i.
$$
For each sequence $J=(\be_1,\dots,\be_l)$ of simple roots, let
$$
\tau_{J}=\tau_{\be_1}\cdots \tau_{\be_l},\quad \deg J=\be_1+\cdots+\be_l.
$$
and for $J=\varnothing$, $\tau_{J}=1$. Then
$$
\tau_i(e_{J}\om_{\mu})=\begin{cases}
                    \frac{q_{ii}}{1-q_{ii}},\qquad\text{if}\  J=(\al_i),\\
                    0,\quad\ \,\qquad\text{otherwise}.
                   \end{cases}
$$
For any $\mu\in Q$, let $k_{\mu}:U_{\q}^{\geq0}\to \bK$ be the algebra homomorphism with
$$
k_{\mu}(xK_{\nu})=\vep(x)q_{\nu\mu}\quad\text{for all}\ \nu\in Q\ \text{and}\ x\in U_{\q}^+.
$$
Then we have for all sequences $J$ of simple roots and all $\mu\in Q$,
$$
k_{\mu}(e_{J}\om_{\nu})=\begin{cases}
                    q_{\nu\mu},\quad\text{if}\  J=\varnothing;\\
                    0,\quad\quad\text{otherwise}.
                   \end{cases}
$$

\begin{lem}{\ }\\
\noindent$(1)$\ For all sequences $J, J'$ of simple roots and all
$\mu\in Q$, we have
$$
\tau_{J}(e_{J'}\om_{\mu})=\tau_{J}(e_{J'})
$$
and if $\deg(J)\neq\deg(J')$, then $\tau_{J}=0.$
\smallskip

\noindent$(2)$\
For all $\mu,\nu\in Q$ and all sequences $J$ of simple roots, we have
$$
k_{\mu}k_{\nu}=k_{\mu+\nu},\qquad
k_{\mu}\tau_{J}=q_{|J|\mu}\tau_{J}k_{\mu}.
$$
\end{lem}

Elements $f_{J}\om'_{\mu}$ with all finite sequences $J$ of simple
roots and $\mu\in Q$ form a basis of $\widetilde{U}_{\q}^{\leq0}$.
Then there is a unique linear map $
\psi:\widetilde{U}_{\q}^{\leq0}\lra (U_{\q}^{\geq0})^* $ with
$\psi(f_{J}\om_{\mu}')=\tau_{J}k_{\mu}$ for all $J$ and $\mu$. Since
\begin{eqnarray*}
&&\psi(f_{J}\om_{\mu}'f_{J'}\om_{\nu}')=q_{|J'|\mu}\psi(f_{J+J'}\om_{\mu+\nu}')
=q_{|J'|\mu}\tau_{J+J'}k_{\mu+\nu},\\
&&\psi(f_{J}\om_{\mu}')\psi(f_{J'}\om_{\nu}')=\tau_{J}k_{\mu}\tau_{J'}k_{\nu}
=q_{|J'|\mu}\tau_{J+J'}k_{\mu+\nu}.
\end{eqnarray*}
We have, for all $J,J'$ and $\mu,\nu$,
$$
\psi(f_{J}\om_{\mu}'f_{J'}\om_{\nu}')=\psi(f_{J}\om_{\mu}')\psi(f_{J'}\om_{\nu}'),
$$
which implies that $\psi$ is in fact an algebra homomorphism. Now we
define a bilinear pairing $ \lg\, ,\, \rg:\
\widetilde{U}_{\q}^{\leq0}\times U_{\q}^{\geq0}\lra \bK $ by
$$
\lg y,\,x\rg=\psi(y)(x)\quad\text{for all }\ y\in \widetilde{U}_{\q}^{\leq0},\
x\in U_{\q}^{\geq0}.
$$
Then we have for all $J,\, J',\, \mu$ and $\nu$,
$$
\lg f_{J'}\om_{\mu}',\,e_{J}\om_{\nu}\rg=\tau_{J'}k_{\mu}(e_{J}\om_{\nu}).
$$
Moreover, we have
$$
\lg y\om_{\mu}',\,x\om_{\nu}\rg=q_{\nu\mu}\lg y,\,x\rg.
$$
and if $\mu,\, \nu\in Q$ with $\mu\neq\nu$, then
$$
\lg y,\,x\rg=0,\quad\text{for all}\ x\in (U_{\q}^+)_{\mu},\ y\in
(\widetilde{U}_{\q}^-)_{-\nu}.
$$

\begin{lem}
For all $x,\, x_1,\, x_2\in U_{\q}^{\geq0}$ and all $y,\, y_1,\,
y_2\in \widetilde{U}_{\q}^{\leq0}$, we have
$$
\lg y_1y_2,x\rg=\lg y_1\ot y_2,\,\De(x)\rg,\qquad  \lg
y,\,x_1x_2\rg=\lg \De(y),\,x_2\ot x_1\rg.
$$
\end{lem}

\begin{lem}\label{rad0}
For all $x\in U_{\q}^{\geq0}$ and  $i\neq j\in I$, we have $\lg
u_{ij}^-,\,x\rg=0.$
\end{lem}
\begin{proof}
It suffices to prove $\lg u_{ij}^-,\,e_J\rg=0$ with $|J|=(1-a_{ij})\al_i+\al_j$ .
We have $J=(\ga,J')$ with $\ga\in\{\al_i,\al_j\}$ where $J'$ is the sequence with $|J'|=|J|-\ga$.
Hence $|J|\neq |J'|$ and $|J'|\neq0$. Then, by Lemma \ref{primi},
\begin{eqnarray*}
\lg u_{ij}^-,\,e_J\rg&=&\lg \Delta(u_{ij}^-),\,e_{J'}\ot e_{\ga}\rg\\
&=&\lg u_{ij}^-\ot \om_i'^{1-a_{ij}}\om_{j}' +1\ot u_{ij}^-,\,e_{J'}\ot e_{\ga}\rg\\
&=&\lg u_{ij}^-,\,e_{J'}\rg\lg \om_i'^{1-a_{ij}}\om_{j}',\,e_\ga\rg+\lg 1,\,e_{J'}\rg\lg u_{ij}^-,\,e_\ga\rg\\
&=&0.
\end{eqnarray*}

This completes the proof.
\end{proof}

\begin{theo}\label{skew1} There exists a unique bilinear pairing
$\lan\,,\,\ran_{\q}:\, U_{\q}^{\leq0}\times U_{\q}^{\geq0}\to \bK$
such that for all $x,\, x'\in U_{\q}^{\geq0}$, $y,\, y'\in
U_{\q}^{\leq0}$, $\mu,\, \nu\in Q$, and $i,\, j\in I$
\begin{eqnarray*}
&&\lg y,\,xx' \rg_{\q}=\lg\De(y),\,x'\ot x\rg_{\q},\\
&&\lg yy',\,x \rg_{\q}=\lg y\ot y',\,\De(x)\rg_{\q},\\
&&\lg f_i,\,e_j\rg_{\q}=\delta_{ij}\frac{q_{ii}}{1-q_{ii}},\\
&& \lg\om_{\mu}',\,\om_{\nu}\rg_{\q}=q_{\nu\mu},\\
&& \lg\om_{\mu}',\,e_i\rg_{\q}=0,\\
&& \lg f_i,\,\om_{\mu}\rg_{\q}=0.
\end{eqnarray*}
\end{theo}
\begin{proof}
Since $U_{\q}^{\leq0}$ is isomorphic to $\widetilde{U}_{\q}^{\leq0}$
modulo the ideal generated by $u_{ij}^-$ for any $i\ne j$, and by
Lemma \ref{rad0}, we have a homomorphism $ \bar{\psi}:\
U_{\q}^{\leq0}\lra (U_{\q}^{\geq0})^*. $ Then we get a bilinear
pairing of $U_{\q}^{\leq0}$ and $U_{\q}^{\geq0}$ via $ \lg
y,x\rg_{\q}=\bar{\psi}(y)(x). $ It is easy to see that the pairing
satisfies all the properties as desired.
\end{proof}

For any two Hopf algebras $ A$ and $ B$ paired by a skew-dual
pairing $\lg\,,\rg$, one may consider the Drinfeld double
construction $ \mathcal{D}(A, B, \lan,\ran)$, which is a Hopf
algebra whose underlying vector space is $ A\otimes B$ with the
tensor product coalgebra structure and the algebra structure defined
by
$$
(a\ot b)(a'\ot b')=\sum \lg S_{B}(b_{(1)}), a'_{(1)}\rg\lg
b_{(3)},a'_{(3)}\rg \,aa'_{(2)}\ot b_{(2)}b',
$$
for $a, a'\in A$ and $b, b'\in B$, and whose antipode $S$ is given
by
$$
S(a\ot b)=(1\ot S_{B}(b))(S_{A}(a)\ot 1).
$$
Therefore we have
\begin{coro}  $U_{\q}(\mg)$
is isomorphic to the Drinfeld double $\mathcal{D}(U_{\q}^{\geq0},U_{\q}^{\leq0},\lan,\ran_{\q})$.
\end{coro}

\subsection{Triangular decomposition}  By the same argument as Coro. 2.6 in
\cite{BGH1}, we have
\begin{coro}\label{tri2}
$U_{\q}(\mg)$ has a triangular decomposition:
\begin{eqnarray*}
U_{\q}(\frak g)\simeq U_{\q}^-\ot U_{\q}^0\ot U_{\q}^+.
\end{eqnarray*}
\end{coro}

\subsection{Hopf $2$-cocycle deformation}
Let $(H,m,\De,1,\vep,S)$ be a Hopf algebra over a field $F$. The
bilinear form $\si: H\times H\to F$ is called a (left) Hopf
$2$-cocycle of $H$ if
\begin{eqnarray}
&&\si(a,1)=\si(1,a)=\vep(a),\qquad\forall\  a\in H,\label{co1}\\
&&\sum \si(a_1,b_1)\si(a_2b_2,c)=\sum
\si(b_1,c_1)\si(a,b_2c_2),\qquad\forall\ a,b,c\in H.\label{co2}
\end{eqnarray}
Let $\si$ be a Hopf $2$-cocycle on $(H,m,\De,1,\vep,S)$, $\si^{-1}$
the inverse of $\si$ under the convolution product. So, by
\cite{DT}, we can construct a new Hopf algebra
$(H^{\si},m^{\si},\De,1,\vep$, $ S^{\si})$, where $H=H^{\si}$ as
coalgebras, and
\begin{eqnarray}
m^{\si}(a\otimes b)=\sum
\si(a_1,b_1)a_2b_2\si^{-1}(a_3,b_3),\quad\forall\; a, \,b\in H,\\
S^{\si}(a)=\sum
\si^{-1}(a_1,S(a_2))S(a_3)\si(S(a_4),a_5),\quad\forall\; a\in H.
\end{eqnarray}
$H$ and $H^{\si}$ are called {\it twisted-equivalent}.

Consider the (standard) one-parameter quantum group
$U_{q,q^{-1}}(\mg_A)$ generated by $E_i, F_i$, $K_i^{\pm1}$ and
$K_i'^{\pm1}(i\in I)$ and satisfying the same relations as those in
Definition \ref{defi} except that
$e_i,f_i,\om_i^{\pm1},\om_i'^{\pm1}$ and $q_{ij}$ are replaced by
$E_i,F_i, K_i^{\pm1},K_i'^{\pm1}$ and $q^{d_ia_{ij}}$, respectively.

Assume $q_{ii}=q^{2d_i}\ (i\in I)$.  Next we shall show that
$U_{q,q^{-1}}(\mg_A)$ is twisted-equivalent to $U_{\q}(\mg_A)$.
\begin{pro}\label{hof}
Let $\si: U_{q,q^{-1}}(\mg_A)\times U_{q,q^{-1}}(\mg_A)\to \bK$ be a
bilinear form on $U_{q,q^{-1}}(\mg_A)$ defined by
$$
\si(x,y)=\begin{cases}&q_{\mu\nu}^{\frac{1}{2}},\,\quad x=K_{\mu}\
\text{ or }\ K_{\mu}',\quad y=K_{\nu}\ \text{ or }\
K_{\nu}',\\
&0,\quad\quad \text{otherwise}.
\end{cases}
$$
Then $\si$ is a Hopf $2$-cocycle of  $U_{q,q^{-1}}(\mg_A)$.
\end{pro}
\begin{proof}
Let $x,y,z$ be any homogenous elements in $U_{q,q^{-1}}(\mg_A)$. If
$x, y, z\in U_{q,q^{-1}}^0$, it is easy to check that the cocycle
conditions (\ref{co1}) and (\ref{co2}) hold. If $x\notin
U_{q,q^{-1}}^0$, then we can assume $\De(x)=a\otimes b+\cdots$ such
that $a\otimes b \notin U_{q,q^{-1}}^0\otimes U_{q,q^{-1}}^0$. Since
$a\notin U_{q,q^{-1}}^0$ and $by_2\notin U_{q,q^{-1}}^0$,
$$
\sum\si(a,y_1)\si(b,y_2z)=0.
$$
Hence, $ \sum\si(x_1,y_1)\si(x_2y_2,z)=0$. Since $x\notin
U_{q,q^{-1}}^0$,
$$
\sum\si(y_1,z_1)\si(x,y_2z_2)=0.
$$
Therefore, $\si$ also satisfies the cocycle conditions (\ref{co1})
and (\ref{co2}). Similarly, if $y$ or $z \notin U_{q,q^{-1}}^0$, we
can show that $\si$ satisfies the cocycle conditions.
\end{proof}

\begin{theo}
Let $\si$ be the Hopf $2$-cocycle defined in Proposition \ref{hof}.
Then we have the following Hopf algebra isomorphism:
$$
U_{\q}(\mg_A)\simeq U_{q,q^{-1}}^\si(\mg_A),
$$
where $U_{q,q^{-1}}^\si(\mg_A)$ is the Hopf algebra via the Hopf
$2$-cocycle deformation of $U_{q,q^{-1}}(\mg_A)$.
\end{theo}
\begin{proof}
Denote $a*b:=m^{\si}(a,b)$ for $a,b\in U_{q,q^{-1}}(\mg_A)$. It
suffices to check the relations:
\begin{eqnarray*}
 &(R^*1)&\quad K_i^{\pm1}*K_j'^{\pm1}=K_j'^{\pm1}*K_i^{\pm1},\quad
K_i^{\pm1}*K_i^{\mp1}=K_i'^{\pm1}K_i'^{\mp1}=1,\\
 &(R^*2)&\quad K_i^{\pm1}*K_j^{\pm1}=K_j^{\pm1}*K_i^{\pm1},
 \quad K_i'^{\pm1}*K_j'^{\pm1}=K_j'^{\pm1}*K_i'^{\pm1},\\
&(R^*3)&\quad K_i*E_j*K_i^{-1}=q_{ij} E_j,
\qquad K_i'*E_j *K_i'^{-1}=q_{ji}^{-1} E_j,
\\
&(R^*4)&\quad K_i*F_j*K_i^{-1}=q_{ij}^{-1} F_j,\qquad
K_i'*F_j *K_i'^{-1}=q_{ji} F_j,\\
&(R^*5)&\quad E_i*F_j-F_j*E_i=\delta_{i,j}\frac{q_{ii}}{q_{ii}-1}({K_i-K_i'}),\\
&(R^*6)&\quad \sum_{k=0}^{1-a_{ij}} (-1)^{k}
\binom{1-a_{ij}}{k}_{q_{ii}}q_{ii}^{\frac{k(k-1)}{2}}
q_{ij}^{k}E_{i}^{*(1-a_{ij}-k)}* E_{j}* E_{i}^{*k} =0 \quad (i\ne j),\\
&(R^*7)&\quad\sum_{k=0}^{1-a_{ij}} (-1)^{k}
\binom{1-a_{ij}}{k}_{q_{ii}}q_{ii}^{\frac{k(k-1)}{2}}
q_{ij}^{k}F_{i}^{k}*F_{j}*F_{i}^{*(1-a_{ij}-k)}=0 \quad (i\ne j).
\end{eqnarray*}
Since
\begin{gather*}
\De^2(K_i)=K_i\ot K_i\ot K_i,\\
\De^2(K_i')=K_i'\ot K_i'\ot K_i',\\
\De^2(E_i)=E_i\ot 1\ot1+K_i\ot E_i\ot 1+K_i\ot K_i\ot E_i,\\
\De^2(F_i)= 1\ot1\ot F_i+1\ot F_i\ot K_i'+F_i\ot K_i'\ot K_i'.
\end{gather*}
It is straightforward to check $(R^*1)$ and $(R^*2)$. For $(R^*3)$
and $(R^*4)$:
\begin{align*}
K_i*E_j&=\si(K_i,K_j)K_iE_j=\si(K_i,K_j)q^{d_ia_{ij}}E_jK_i\\
&=\si(K_i,K_j)q^{d_ia_{ij}}\si(K_j,K_i)^{-1}E_j*K_i\\
&=q_{ij}^{\frac{1}{2}}q^{d_ia_{ij}}q_{ji}^{-\frac{1}{2}}E_j*K_i\\
&=q_{ij}(q_{ij}q_{ji})^{-\frac{1}{2}}q^{d_ia_{ij}}E_j*K_i\\
&=q_{ij}(q_{ii})^{-\frac{a_{ij}}{2}}q^{d_ia_{ij}}E_j*K_i\\
&=q_{ij}q^{-d_ia_{ij}}q^{d_ia_{ij}}E_j*K_i\\
&=q_{ij}E_j*K_i,
\end{align*}
\begin{align*}
K_i'*E_j&=\si(K_i',K_j)K_i'E_j=\si(K_i',K_j)q^{-d_ia_{ij}}E_jK_i'\\
&=\si(K_i',K_j)q^{-d_ia_{ij}}\si(K_j,K_i')^{-1}E_j*K_i'\\
&=q_{ij}^{\frac{1}{2}}q^{-d_ia_{ij}}q_{ji}^{-\frac{1}{2}}E_j*K_i'\\
&=q_{ji}^{-1}(q_{ij}q_{ji})^{\frac{1}{2}}q^{-d_ia_{ij}}E_j*K_i'\\
&=q_{ji}^{-1}(q_{ii})^{\frac{a_{ij}}{2}}q^{-d_ia_{ij}}E_j*K_i'\\
&=q_{ji}^{-1}q^{d_ia_{ij}}q^{-d_ia_{ij}}E_j*K_i'\\
&=q_{ji}^{-1}E_j*K_i',
\end{align*}
\begin{align*}
K_i*F_j&=\si(K_i,K_j')^{-1}K_iF_j=\si(K_i,K_j')^{-1}q^{-d_ia_{ij}}F_jK_i\\
&=\si(K_i,K_j')^{-1}q^{-d_ia_{ij}}\si(K_j',K_i)F_j*K_i\\
&=q_{ij}^{-\frac{1}{2}}q^{-d_ia_{ij}}q_{ji}^{\frac{1}{2}}F_j*K_i\\
&=q_{ij}^{-1}(q_{ij}q_{ji})^{\frac{1}{2}}q^{-d_ia_{ij}}F_j*K_i\\
&=q_{ij}^{-1}(q_{ii})^{\frac{a_{ij}}{2}}q^{-d_ia_{ij}}F_j*K_i\\
&=q_{ij}^{-1}q^{d_ia_{ij}}q^{-d_ia_{ij}}F_j*K_i\\
&=q_{ij}^{-1}F_j*K_i,
\end{align*}
\begin{align*}
K_i'*F_j&=\si(K_i',K_j')^{-1}K_i'F_j=\si(K_i',K_j')^{-1}q^{d_ia_{ij}}F_jK_i'\\
&=\si(K_i',K_j')^{-1}q^{d_ia_{ij}}\si(K_j',K_i')F_j*K_i'\\
&=q_{ij}^{-\frac{1}{2}}q^{d_ia_{ij}}q_{ji}^{\frac{1}{2}}F_j*K_i'\\
&=q_{ji}(q_{ij}q_{ji})^{-\frac{1}{2}}q^{d_ia_{ij}}F_j*K_i'\\
&=q_{ji}(q_{ii})^{-\frac{a_{ij}}{2}}q^{d_ia_{ij}}F_j*K_i'\\
&=q_{ji}q^{-d_ia_{ij}}q^{d_ia_{ij}}F_j*K_i'\\
&=q_{ji}F_j*K_i'.
\end{align*}
For $(R^*5)$:
$$
E_i*F_j-F_j*E_i=E_iF_j-F_jE_i=\delta_{i,j}\frac{q_{ii}}{q_{ii}-1}(K_i-K_i').
$$
For $(R^*6)$:
$$
E_{i}^{*(1-a_{ij}-k)}*E_{j}*E_{i}^{*k} =q_{ii}^{\frac{(a_{ij}-1)a_{ij}}{4}}q_{ij}^{\frac{1-a_{ij}-k}{2}}q_{ji}^{\frac{k}{2}}
E_{i}^{1-a_{ij}-k} E_{j}E_{i}^{k}.
$$
Hence
\begin{eqnarray*}
&&\sum_{k=0}^{1-a_{ij}} (-1)^{k}
\binom{1-a_{ij}}{k}_{q_{ii}} q_{ii}^{\frac{k(k-1)}{2}}q_{ij}^k
E_{i}^{*(1-a_{ij}-k)}*E_{j}*
E_{i}^{*k} \\
&=&\sum_{k=0}^{1-a_{ij}} (-1)^{k} \binom{1-a_{ij}}{k}_{q_{ii}}
q_{ii}^{\frac{k(k-1)}{2}}q_{ij}^k
q_{ii}^{\frac{(a_{ij}-1)a_{ij}}{4}}q_{ij}^{\frac{1-a_{ij}-k}{2}}q_{ji}^{\frac{k}{2}}
E_{i}^{1-a_{ij}-k}
E_{j}E_{i}^{k} \\
&=&q_{ii}^{\frac{(a_{ij}-1)a_{ij}}{4}}q_{ij}^{\frac{1-a_{ij}}{2}}\sum_{k=0}^{1-a_{ij}}
(-1)^{k} \binom{1-a_{ij}}{k}_{q_{ii}} q_{ii}^{\frac{k(k-1)}{2}}
q_{ij}^{\frac{k}{2}}q_{ji}^{\frac{k}{2}} E_{i}^{1-a_{ij}-k}
E_{j}E_{i}^{k} \\
&=&q_{ii}^{\frac{(a_{ij}-1)a_{ij}}{4}}q_{ij}^{\frac{1-a_{ij}}{2}}\sum_{k=0}^{1-a_{ij}} (-1)^{k}
\binom{1-a_{ij}}{k}_{q_{ii}} q_{ii}^{\frac{k(k-1)}{2}}
(q_{ij}q_{ji})^{\frac{k}{2}}
E_{i}^{1-a_{ij}-k}
E_{j}E_{i}^{k} \\
&=&q_{ii}^{\frac{(a_{ij}-1)a_{ij}}{4}}q_{ij}^{\frac{1-a_{ij}}{2}}\sum_{k=0}^{1-a_{ij}} (-1)^{k}
\binom{1-a_{ij}}{k}_{q^{2d_i}} q^{d_ik(k-1+a_{ij})}
E_{i}^{1-a_{ij}-k}
E_{j}E_{i}^{k} \\
&=&0.
\end{eqnarray*}
For $(R^*7)$:
Since
$$
F_{i}^{*k}*F_{j}*F_{i}^{*(1-a_{ij}-k)}=q_{ii}^{\frac{(1-a_{ij})a_{ij}}{4}}q_{ji}^{-\frac{1-a_{ij}-k}{2}}q_{ij}^{-\frac{k}{2}}
F_{i}^{k}F_{j}F_{i}^{1-a_{ij}-k} .
$$
Therefore,
\begin{eqnarray*}
&&\sum_{k=0}^{1-a_{ij}} (-1)^{k}
\binom{1-a_{ij}}{k}_{q_{ii}} q_{ii}^{\frac{k(k-1)}{2}}q_{ij}^k\,
F_{i}^{*k}*F_{j}*F_{i}^{*(1-a_{ij}-k)}\\
&=&\sum_{k=0}^{1-a_{ij}} (-1)^{k}
\binom{1-a_{ij}}{k}_{q_{ii}} q_{ii}^{\frac{k(k-1)}{2}}q_{ij}^kq_{ii}^{\frac{(1-a_{ij})a_{ij}}{4}}q_{ji}^{-\frac{1-a_{ij}-k}{2}}q_{ij}^{-\frac{k}{2}}\,
F_{i}^{k}
F_{j}F_{i}^{1-a_{ij}-k}\\
&=&q_{ii}^{\frac{(1-a_{ij})a_{ij}}{4}}q_{ji}^{\frac{a_{ij}-1}{2}}\sum_{k=0}^{1-a_{ij}} (-1)^{k}
\binom{1-a_{ij}}{k}_{q_{ii}} q_{ii}^{\frac{k(k-1)}{2}}q_{ij}^{\frac{k}{2}}q_{ji}^{\frac{k}{2}}
F_{i}^{k}
F_{j}F_{i}^{1-a_{ij}-k}\\
&=&q_{ii}^{\frac{(1-a_{ij})a_{ij}}{4}}q_{ji}^{\frac{a_{ij}-1}{2}}\sum_{k=0}^{1-a_{ij}} (-1)^{k}
\binom{1-a_{ij}}{k}_{q_{ii}} q_{ii}^{\frac{k(k-1)}{2}}(q_{ij}q_{ji})^{\frac{k}{2}}
F_{i}^{k}
F_{j}F_{i}^{1-a_{ij}-k}\\
&=&q_{ii}^{\frac{(1-a_{ij})a_{ij}}{4}}q_{ji}^{\frac{a_{ij}-1}{2}}\sum_{k=0}^{1-a_{ij}}
(-1)^{k} \binom{1-a_{ij}}{k}_{q_{ii}}
q_{ii}^{\frac{k(k-1)}{2}}q_{ii}^{\frac{a_{ij}k}{2}} F_{i}^{k}
F_{j}F_{i}^{1-a_{ij}-k}\\
&=&q_{ii}^{\frac{(1-a_{ij})a_{ij}}{4}}q_{ji}^{\frac{a_{ij}-1}{2}}\sum_{k=0}^{1-a_{ij}} (-1)^{k}
\binom{1-a_{ij}}{k}_{q^{2d_i}} q^{d_ik(k-1+a_{ij})}
F_{i}^{k}
F_{j}F_{i}^{1-a_{ij}-k}\\
&=&0.
\end{eqnarray*}

The proof is complete.
\end{proof}

\section{Representation Theory}
When $\mg_A$ is of finite type, we denote
$$
q_{\mu\nu}=\prod_{i,j\in I}q_{ij}^{\mu\nu}
$$
for $\mu=\sum_{i\in I}\mu_i\al_i, \ \nu=\sum_{i\in
I}\nu_i\al_i\in\La$. When $\mg_{A}$ is of affine type, let
$I=\{0,1,\cdots,l\}$ and $\La=\sum_{i\in I}\bZ \La_i $ such that
$\La_i(h_j)=\de_{i,j}$ for $i,\, j\in I$, where $\La_i$ is the $i$th
fundamental weight of $\mg_A$. Let $q_{\La_0\al_i},\
q_{\al_i\La_0}\in\bK \ (i\in I)$ such that
\begin{eqnarray}
q_{\La_0\al_i}q_{\al_i\La_0}=q_{ii}^{\de_{i,0}},\quad\forall \ i\in I.\label{bu}
\end{eqnarray}
Now we can define $q_{\mu\nu}$ for $\mu,\nu\in\La$ as above.

\subsection{Category $\O_{int}^{\q}$}
\begin{defi}
The category $\O_{int}^{\q}$  consists of $U_{\q}(\mg_A)$-modules
$V^{\q}$ with the following conditions satisfied:

$(1)$\ $V^{\q}$ has a weight space decomposition
$V^{\q}=\bigoplus_{\la\in\La}V^{\q}_{\la}$, where
$$
V^{\q}_{\la}=\{v\in V^{\q}\mid\om_{i}v=q_{\al_i\la}v,\
\om_{i}'v=q_{\la\al_i}^{-1}v, \ \forall\; i\in I\}
$$
and $\dim V_{\la}^{\q}<\infty$ for all $\la\in\La$.

$(2)$\ There exist a finite number of elements $\la_1,\dots,
\la_t\in \La$ such that
$$
\wt (V^{\q})\subset D(\la_1)\cup\cdots\cup D(\la_t),
$$
where $D(\la_i):=\{\mu\in\La\,|\,\mu<\la_i\}$.

$(3)$\  $e_i$ and $f_i$ are locally nilpotent on $V^{\q}$.

The morphisms are taken to be usual $U_{\q}(\mg_A)$-module
homomorphisms.
\end{defi}

\begin{lem}\label{key} For any $\lambda\in\Lambda$, and $i\in I$, we have
\begin{equation}
q_{\al_i\la}q_{\la\al_i}=q_{ii}^{\la(h_i)}.
\end{equation}
\end{lem}
\begin{proof}
It suffices to prove
$$q_{\al_i\La_j}q_{\La_j\al_i}=q_{ii}^{\La_j(h_i)}
=q_{ii}^{\de_{ij}},\quad \forall\; i,\, j\in I.$$
By (\ref{bu}), let $\la_j=\sum_{k\in I}x_{kj}\al_k$. Then
\begin{eqnarray*}
q_{\al_i\la_j}q_{\la_j\al_i}&=&\prod_{k\in I}q_{ik}^{x_{kj}}\prod_{k\in I}q_{ki}^{x_{kj}}
=\prod_{k\in I}(q_{ik}q_{ki})^{x_{kj}}\\
&=&
(q_{ii})^{\sum_{k\in I}a_{ik}x_{kj}}\\
&=&q_{ii}^{\delta_{ij}}.
\end{eqnarray*}

This completes the proof.
\end{proof}

\begin{lem}\label{com}
For any $i\in I$, $m\in\bZ$ and $m\geq1$, we have
\begin{align}
e_if_i^m
&=f_i^me_i+\frac{q_{ii}}{q_{ii}-1}f_i^{m-1}\left((m)_{q_{ii}^{-1}}\om_i
-(m)_{q_{ii}}\om_i'\right),\\
e_i^mf_i&=f_ie_i^m+\frac{q_{ii}}{q_{ii}-1}e_i^{m-1}
\left((m)_{q_{ii}}\om_i-(m)_{q_{ii}^{-1}}\om_i'\right).
\end{align}
\end{lem}
\begin{proof}
For $m = 1$, it is the relation $(R6)$. For $m>1$, we have
$$
e_if_i^m
=f_i^me_i+\frac{q_{ii}}{q_{ii}-1}f_i^{m-1}\left((m)_{q_{ii}^{-1}}\om_i
-(m)_{q_{ii}}\om_i'\right).
$$
Then
\begin{eqnarray*}
e_if_i^{m+1} &=&f_i^me_if_i+\frac{q_{ii}}{q_{ii}-1}f_i^{m-1}
\left((m)_{q_{ii}^{-1}}\om_i-(m)_{q_{ii}}\om_i'\right)f_i\\
&=&f_i^m\left(f_ie_i+\frac{q_{ii}}{q_{ii}-1}(\om_i-\om_i')\right) +
\frac{q_{ii}}{q_{ii}-1}f_i^{m}\left(q_{ii}^{-1}(m)_{q_{ii}^{-1}}\om_i
 -q_{ii}(m)_{q_{ii}}\om_i'\right)\\
&=&f_i^{m+1}e_i
 +\frac{q_{ii}}{q_{ii}-1}f_i^{m}\left((m+1)_{q_{ii}^{-1}}\om_i-(m+1)_{q_{ii}}\om_i'\right).
\end{eqnarray*}
Similarly, the second equation holds.
\end{proof}

For each $i\in I$, let $U_i$ be a subalgebra of  $U_{\q}(\mg_A)$
generated by $e_i,\, f_i, \,\om_i^{\pm1},\, \om_i'^{\pm1}$.

\begin{pro}\label{sl22}  \
Let $\phi:U_i^0\to \bK$ be a homomorphism of algebras. Denote
$$\phi_i:=\phi(\om_i),\quad\phi'_i:=\phi(\om'_i),\quad
v_j:=f^j\otimes v_\phi\in M(\phi),\ j\ge 0.$$
Then

$(\text{\rm i})$ \ $M(\phi)$ is a simple $U_i$-module if and only if
$\phi_i- q_{ii}^{-j}\phi'_i\ne 0,\ \forall\  j\ge 0$.

$(\text{\rm ii})$ \ If $\phi_i'=\phi_iq_{ii}^{-m}$ for $m\ge 0$,
then $M(\phi)$ has a unique maximal submodule
$$
N=\text{\rm Span}_\mathbb K\{\,v_j\mid j\ge
m+1\,\}\cong M(\phi-(m+1)\al_i).
$$

$(\text{\rm iii})$\ The simple $U_i$-module $L(\phi)$ is
$(m+1)$-dimensional. Moreover, it is spanned by $v_0,v_1,\cdots,v_m$
such that
\begin{gather*}
\om_i.v_j=\phi_i q_{ii}^{-j}v_j,\\
\om'_i.v_j=\phi_i q_{ii}^{j-m}v_j,\\
f_i.v_j=v_{j+1},\ (v_{m+1}=0),\\
e_i.v_j=\phi_i q_{ii}^{-m+1}(m-j+1)_{q_{ii}}(j)_{q_{ii}}v_{j-1},
\quad (v_{-1}=0).
\end{gather*}

$(\text{\rm iv})$\ Any $(m+1)$-dimensional simple $U_i$-module is
isomorphic to $L(\phi)$ for some $\phi$.

$(\text{\rm v})$ \ Let $\nu=\sum_{i\in
I}\nu_i\Lambda_i\in\Lambda^+$. Then $\hat
\nu(\om_i')=\hat\nu(\om_i)q_{ii}^{-\nu_i}$ and $U_i$-module
$L(\nu_i\Lambda_i)$ is $(\nu_i+1)$-dimensional and
$\phi_i=\hat\nu(\om_i)$. Here $\hat{\nu}: U^0\to \bK$ is the algebra
homomorphism such that $\hat{\nu}(\om_i)=q_{\al_i\mu}$,
$\hat{\nu}(\om_i')=q_{\mu£¬\al_i}^{-1},\ \forall\ i\in I$.
\end{pro}
\begin{proof} Similar to the argument of two-parameter cases (see
\cite{BGH2}), in particular, for $(\text{\rm v})$, by Lemma
\ref{key}, we have
$$
\frac{\hat\nu(\om_i')}{\hat\nu(\om_i)}=q_{\La_i\nu}^{-1}q_{\nu\La_i}^{-1}=
q_{ii}^{-\nu(h_i)}=q_{ii}^{-\nu_i}=\frac{\widehat{\nu_i\la_i}(\om_i')}{\widehat{\nu_i\la_i}(\om_i)},
\quad \forall\ i\in I.
$$

The proof is complete.
\end{proof}

\begin{pro}\label{iteg1}
Let $\la\in \La^+$. Let $V^{\q}(\la)$ be an irreducible highest
module with highest weight vector $v_{\la}$. Then
$$f_i^{\la(h_i)+1}v_{\la}=0,\quad\forall\ i\in I.$$
\end{pro}
\begin{proof}
By Lemma \ref{com},
\begin{align*}
e_if_i^m.v_{\la}&=((m)_{q_{ii}^{-1}}q_{\al_i\la}-(m)_{q_{ii}}q_{\la\al_i}^{-1})f_i^{m-1}.v_{\la}\\
&=(m)_{q_{ii}}q_{\la\al_i}^{-1}(q_{ii}^{-m+1}q_{\al_i\la}q_{\la\al_i}-1)f_i^{m-1}.v_{\la}.
\end{align*}
By Lemma \ref{key},
$$
e_if_i^{\la(h_i)+1}.v_{\la}=0.
$$
By Lemma \ref{com},
$$e_jf_i^{\la(h_i)+1}.v_{\la}=0,\quad\forall j\neq i. $$
If $f_i^{\la(h_i)+1}.v_{\la}\neq0$, then there exists a nontrivial
submodule, contradicting the irreducibility of $V^{\q}(\la)$.
\end{proof}

\begin{coro}\label{em}  Let $\la\in \La^+$.  Let $V^{\q}(\la)$ be an irreducible highest
module with highest weight vector $v_{\la}$. Let $\beta=\sum_{i\in
I}m_i\al_i\in Q^+$ such that $\la(h_i)\geq m_i$, $\forall\; i\in I$.
Then for any $x\in (U_{\q}^-)_{-\be}$, the map $x\mapsto
x.v_\lambda$ is injective.
\end{coro}

\begin{pro}\label{iteg2}
Let $V^{\q}(\la)$ be an irreducible highest module with highest
weight vector $v_{\la}$. Then $V^{\q}(\la)$ is integrable if and
only if for every $i\in I$, there exists some $N_i$ such that
$f_i^{N_i}.v_{\la}=0.$
\end{pro}
\begin{proof}
It is clear that $e_i \ (i\in I)$ are locally nilpotent on any
highest weight module. It suffices to show that $f_i \ (i\in I)$ are
locally nilpotent on $V^{\q}(\la)$. Let $j\neq i$. We shall show
that for $N\geq 1-a_{ij}$,
\begin{equation}\label{com1}
f_i^{N}f_j\in\sum_{m+n=-a_{ij},N+a_{ij}\leq n\leq N}\bK f_i^{m}f_jf_i^{n}.
\end{equation}
For $N=1-a_{ij}$, it is just $\q$-Serre relation $(R7)$. Assume for $N\geq 1-a_{ij}$,
the claim holds. For
$N+1$, by induction,
$$
f_i^{N+1}f_j\in\sum_{m+n=-a_{ij},N+a_{ij}\leq n\leq N}\bK f_i^{m+1}f_jf_i^{n}.
$$
By $\q$-Serre relation $(R7)$,
$$
f_i^{1-a_{ij}}f_jf_i^{N+a_{ij}}\in\sum_{s+t=1-a_{ij},
1\leq t\leq1-a_{ij}}\bK f_i^{s}f_jf_i^{t+N+a_{ij}}.
$$
Then (\ref{com1}) holds. For a sufficiently large $N$, $f_i^Ny\in
U_{\q}^-f_i^{N_i},\ y\in U_{\q}^-$. Note that every element of
$V^{\q}$ can be written in the form $yv_{\la},\ y\in U_{\q}^-$. This
completes the proof.
\end{proof}

\begin{pro}\label{dom}
Let $V^{\q}(\la)$ be an irreducible highest module with highest
weight vector $v_{\la}$. Then $V^{\q}(\la)$ belongs to category
$\O^{\q}_{int}$ if and only if $\la\in \La^+.$
\end{pro}

\begin{proof}
By Propositions \ref{iteg1} and \ref{iteg2}, we get the ``if" part.
Now we shall show the ``only if" part. It suffices to prove
$(\lambda,\al_i^\vee)\ge 0$ for any $i\in I$. Since $f_i $ is
locally nilpotent on $V^{\q}(\la)$, there exists some $m_i\ge 0$
such that $f_i^{m_i+1}.v=0$ and $f_i^{m_i}.v\ne 0$ for $i\in I$. By
Lemma \ref{com} and $e_i.v=0$, we have
\begin{eqnarray*}
0=e_if_i^{m_i+1}.v
&=&f_i^{m_i+1}e_i.v+\frac{q_{ii}}{q_{ii}-1}f_i^{m_i}
\left((m_i+1)_{q_{ii}^{-1}}\om_i-(m_i+1)_{q_{ii}}\om_i'\right)v\\
&=&\frac{q_{ii}}{q_{ii}-1}f_i^{m_i}.v\left((m_i+1)_{q_{ii}^{-1}}q_{\al_i\la}-
(m_i+1)_{q_{ii}}q_{\la\al_i}^{-1}\right).
\end{eqnarray*}
Hence $q_{ii}^{m_i}=q_{\al_i\la}q_{\la\al_i}.$ With the help of
Lemma $\ref{key}$, we have $q_{ii}^{m_i}=q_{ii}^{\la(h_i)}$. Since
$q_{ii}\,(i\in I)$ are not roots of unity, $\la(h_i)=m_i$.
\end{proof}

\begin{lem}\label{zero}{\ }

$(1)$\  Let $y\in
(U_{\q}^-)_{-\be}$ such that $[e_i,y]=0$ for all $i\in I$. Then $x=0$.

$(2)$\  Let $x\in
(U_{\q}^+)_{\be}$ such that $[f_i,x]=0$ for all $i\in I$. Then $x=0$.
\end{lem}
\begin{proof}
Let $y\in (U_{\q}^-)_{-\be}$ such that $[e_i,y]=0$ for all $i\in I$.
By Corollary \ref{em}, we can choose a sufficiently large
$\la\in\La^{+}$ such that
$$
(U_{\q}^-)_{-\beta}\lra V^{\q}(\la),\quad
u\mapsto u.v_{\la}
$$
is injective. Here  $V^{\q}(\la)$ is an irreducible highest module
with highest weight vector $v_{\la}$. $yv_{\lambda}$ generates a
submodule of $V^{\q}(\la)$. Since $V^{\q}(\la)$ is irreducible,
 $yv_{\la}=0$, which implies $y=0$. Using the anti-automorphism $\Psi$ of $U_{\q}(\mg)$ in Lemma
\ref{auto}, we can prove (2) directly.
\end{proof}

\subsection{Skew derivations}
By coproduct, we have
$$\Delta(x)\in\sum_{0\le\nu\le\be}(U^+_{\q})_{\be-\nu}\om_\nu\otimes
(U^+_{\q})_{\nu},\quad\text{for all}\  x\in
(U^+_{\q})_{\be},$$
For $i\in I$ and $\be\in Q^+$, we can define the skew-derivations
$${\hat\pa_i},\,{_i\hat\pa}: \,(U^+_{\q})_{\be}\lra
(U^+_{\q})_{\be-\al_i}$$ such that
\begin{equation*}
\begin{split}
\Delta(x)&=x\otimes 1+\sum_{i\in I}\hat\pa_i(x)\,\om_i\otimes
e_i+\text{the rest},\\
\Delta(x)&=\om_\be\otimes
x+\sum_{i\in I}e_i\,\om_{\be-\al_i}\otimes\,_i{\hat\pa}(x)+\text{
 the rest},
\end{split}
\end{equation*}
where in each case ``the rest" refers to terms involving products of
more than one $e_j$ in the second (resp. first) factor. Let
$$
\pa_i:=\frac{q_{ii}}{1-q_{ii}}\hat\pa_i,\quad {_i{\pa}}:=\frac{q_{ii}}{1-q_{ii}}{_i\hat{\pa}}
$$

\begin{lem}\label{sk} \ For all $x\in (U^+_{\q})_{\be}$, $x'\in (U^+_{\q})_{\be'}$, and $y\in
U_{\q}^-$, we have the following relations:

$(\text{\rm i})$ \ \;
$\partial_i(xx')=q_{\al_i\be'}\pa_i(x)\,x'+x\,\partial_i(x')$,

$(\text{\rm ii})$ \;
$_i\partial(xx')=\,_i\partial(x)\,x'+q_{\be\al_i}x\,_i\partial(x')$,

$(\text{\rm iii})$ \ $\lg f_iy,\, x\rg_{\q} =\lg y,\,
_i\partial(x)\rg_{\q}$,

$(\text{\rm iv})$ \ $\lg yf_i,\, x\rg_{\q} =\lg y,\,
\partial_i(x)\rg_{\q}$,

$(\text{\rm v})$ \ \,
$f_ix-xf_i=\partial_i(x)\,\om_i-\om_i'\,_i\partial(x)$.
\end{lem}
\begin{proof}
It is straightforward to check.
\end{proof}

\begin{pro}\label{non} For each $\beta\in Q^+$, the restriction of
pairing $\lg\, ,\, \rg_{\q}$ to $(U^-_{\q})_{-\be}\times
(U^+_{\q})_{\be}$ is nondegenerate.
\end{pro}
\begin{proof}
We use induction on $\beta$ with respect to the usual partial order:
$\beta'\le\beta $ if $\beta-\beta'\in Q^{+}$. The claim holds for
$\beta=0$, since $\lg1,1\rg_{\q}=1$. Assume that $\beta\ge 0$, and
 the claim holds for all $\al$ with $0\le\al<\beta$. Let
$x\in (U^+_{\q})_{\be}$ with $\lg y,x\rg_{\q}=0$ for all $y\in
(U^-_{\q})_{\be}$. In particular, we have for all $y\in
(U^-_{\q})_{-\beta+\al_i}$ that
$$
\lg f_iy,x\rg_{\q}=0,\qquad\lg yf_i,x\rg_{\q}=0, \qquad i\in I.
$$
It follows from Lemma \ref{sk} (iii) and (iv) that
$$\lg f_iy,x\rg_{\q}=\lg y,\,
_i\partial(x)\rg_{\q}=0,\qquad\lg yf_i,x\rg_{\q}=\lg y,\,
\partial_i(x)\rg_{\q}=0.$$ By the induction hypothesis, we have
$_i\partial(x)=\partial_i(x)=0$, and it follows from Lemma \ref{sk}
(v) that $f_ix=xf_i$ for all $i$. Now Lemma \ref{zero} applies to
give $x=0$, as desired.
\end{proof}

\subsection{}
By Proposition \ref{non}, we can take a basis
$\{u_k^\be\}_{k=1}^{d_\be}$, ($d_\be =\dim (U_{\q}^+)_{\be}$) of
$(U_{\q}^+)_\be$, and the dual basis $\{v_k^\be\}_{k=1}^{d_\be}$ of
$(U_{\q}^-)_{-\be}$. Then, for any $x\in (U_{\q}^+)_{\be}$ and $y\in
(U_{\q}^-)_{-\be}$,
\begin{equation}
x=\sum_{k=1}^{d_\be}\lan v_k^\be, x\ran_{\q}\,u_k^\be, \qquad
y=\sum_{k=1}^{d_\be}\lan y, u_k^\be\ran_{\q}\,v_k^\be.
\end{equation}
For $\be\in Q^+$, let
\begin{equation}
\Theta_\be=\sum_{k=1}^{d_\be}v_k^{\be}\otimes u_k^{\be}.
\end{equation}
Set $\Theta_\be=0$ if $\be\not\in Q^+$.
\begin{equation}
\Theta=\sum_{\be\in Q^+}\Theta_\be.
\end{equation}

\begin{lem}\label{rm} \ For $i\in I$, $\be\in Q^+$,

$(\text{\rm i})$ \ \;
$(\om_i\otimes\om_i)\,\Theta_\be=\Theta_\be\,(\om_i\otimes\om_i)$,
\qquad
$(\om_i'\otimes\om_i')\,\Theta_\be=\Theta_\be\,(\om_i'\otimes\om_i'),$

$(\text{\rm ii})$ \ \ $(e_i\otimes 1)\,\Theta_\be+(\om_i\otimes
e_i)\,\Theta_{\be-\al_i}=\Theta_\be\,(e_i\otimes
1)+\Theta_{\be-\al_i}\,(\om_i'\otimes e_i),$

$(\text{\rm iii})$ \ \,$(1\otimes f_i)\,\Theta_\be+(f_i\otimes
\om_i')\,\Theta_{\be-\al_i}=\Theta_\be\,(1\otimes
f_i)+\Theta_{\be-\al_i}\,(f_i\otimes \om_i).$\hfill\qed
\end{lem}

Let ${\Om}^{\q}_\be=\sum_kS(v_k^{\be})u_k^\be$, where $S$ is the
antipode. The quantum Casimir operator ${\Om}^{\q}$ can be defined
\begin{equation}
{\Om}^{\q}:=\sum_{\be\in Q^+}{\Om}^{\q}_\be=\sum_{\be\in Q^+}\sum_{k}S(v_k^{\be})u_k^\be.
\end{equation}
Note that ${\Om}^{\q}$ is well-defined.

\begin{lem}\label{key5}
Let $\psi$ be the automorphism of $U_{\q}(\mg_A)$ defined by
$$\psi(\om_i)=\om_i,\quad\psi(\om_i')=\om_i',\quad \psi(e_i)=\om_i'\om_i^{-1}e_i,
\quad\psi(f_i)=f_i{\om_i'}^{-1}\om_i.$$ Then
\begin{equation*}
\psi(x){\Om}^{\q}={\Om}^{\q} x,\quad\forall\, x\in U_{\q}.
\end{equation*}
\end{lem}
\begin{proof}
It is straightforward to check.
\end{proof}

\begin{coro} For any $V\in Ob(\O^{\q}_{int})$ and
$v\in V_\la$, we have
\begin{equation}
{\Om}^{\q}\,e_iv=q_{ii}^{-(\la+\al_i)(h_i)}e_i\,{\Om}^{\q} v,\quad
{\Om}^{\q}\,f_i v=q_{ii}^{\la(h_i)}f_i\,{\Om}^{\q} v.
\end{equation}
\end{coro}

\begin{proof} For any $v\in V_\lambda$ and $i\in I$, by Lemma \ref{key5},
\begin{align*}
&\psi(e_i)\,{\Om}^{\q}. v=\om_i'\om_i^{-1}e_i {\Om}^{\q}. v
=q_{\al_i,\la+\al_i}^{-1}q_{\la+\al_i,\al_i}^{-1}e_i{\Om}^{\q}. m=q_{ii}^{-(\la+\al_i)(h_i)}e_i{\Om}^{\q}. v,\\
&\psi(f_i)\,{\Om}^{\q}. v=f_i{\om_i'}^{-1}\om_i{\Om}^{\q}. v=q_{\al_i\la}q_{\la\al_i}f_i{\Om}^{\q}. v
=q_{ii}^{\la(h_i)}f_i{\Om}^{\q}. v.
\end{align*}
This completes the proof.
\end{proof}

Note that the following fact:
$$q_{ij}q_{ji}=q_{ii}^{a_{ij}}=q_{jj}^{a_{ji}}=q_{ji}q_{ij},\quad\forall\ i,j\in I,$$
and
$$d_ia_{ij}=d_ja_{ji},\quad\forall\  i,j\in I.$$
Then
\begin{equation}
q_{ii}^{\frac{1}{d_i}}=q_{jj}^{\frac{1}{d_j}},\quad \forall\ i,
\;j\in I.
\end{equation}
Let $t=q_{ii}^{\frac{1}{d_i}},\ \forall\ i\in I$.
For $V^{\q}\in Ob(\O^{\q}_{int})$, we can define
$${\Xi}^{\q}:\, V^{\q}\lra V^{\q}$$
such that
\begin{equation}
{\Xi}^{\q}\,v_\mu=g(\mu)v_\mu, \quad
\text{for } \ v_\mu\in V_\mu^{\q},\ i\in I,
\end{equation}
where $g(\mu)=t^{\frac{(\mu+\rho,\mu+\rho)}{2}}$.

\begin{pro}\label{commute} For $V\in Ob(\O^{\q}_{int})$,
then the action of ${\Om}^{\q}\cdot{\Xi}^{\q}:\,V^{\q}\lra
V^{\q}$ commutes with the action of $U_{\q}$ on $V$.
\end{pro}
\begin{proof} It suffices to check the result on generators. Then for $v\in V_{\mu}^{\q}$ and
$i\in I$, we have
\begin{eqnarray*}
{\Om}^{\q}\cdot{\Xi}^{\q} (e_i.v)&=&g(\mu+\al_i){\Om}^{\q} e_i. v\\
&=&g(\mu+\al_i)q_{ii}^{-(\mu+\al_i)(h_i)}e_i{\Om}^{\q}.  v\\
&=&g(\mu+\al_i)g(\mu)^{-1}q_{ii}^{-(\mu+\al_i)(h_i)}e_i{\Om}^{\q}{\Xi}^{\q}.v\\
&=&t^{\frac{(\mu+\al_i+\rho,\mu+\al_i+\rho)}{2}}t^{-\frac{(\mu+\rho,\mu+\rho)}{2}}q_{ii}^{-\frac{(\mu+\al_i,\al_i)}{d_i}}e_i{\Om}^{\q}{\Xi}^{\q}.v\\
&=&t^{(\mu+\rho,\al_i)+d_i}t^{-(\mu+\al_i,\al_i)}e_i{\Om}^{\q}\cdot{\Xi}^{\q}.v\\
&=&t^{(\mu,\al_i)+2d_i}t^{-(\mu,\al_i)-2d_i}e_i{\Om}^{\q}\cdot{\Xi}^{\q}.v\\
&=&e_i{\Om}^{\q}\cdot{\Xi}^{\q}.v.
\end{eqnarray*}
Moreover,
\begin{eqnarray*}
{\Om}^{\q}\cdot{\Xi}^{\q} (f_i.v)&=&g(\mu-\al_i){\Om}^{\q} f_i.v\\
&=&g(\mu-\al_i)q_{ii}^{\mu(h_i)}f_i{\Om}^{\q}.v\\
&=&g(\mu-\al_i)g(\mu)^{-1}q_{ii}^{\mu(h_i)}f_i{\Om}^{\q}\cdot{\Xi}^{\q}.v\\
&=&t^{\frac{(\mu-\al_i+\rho,\mu-\al_i+\rho)}{2}}t^{-\frac{(\mu+\rho,\mu+\rho)}{2}}
q_{ii}^{\frac{(\mu,\al_i)}{d_i}}f_i{\Om}^{\q}\cdot{\Xi}^{\q}.v\\
&=&t^{-(\mu,\al_i)}t^{(\mu,\al_i)}f_i{\Om}^{\q}\cdot{\Xi}^{\q}.v\\
&=&f_i{\Om}^{\q}\cdot{\Xi}^{\q}.v.
\end{eqnarray*}

We complete the proof.
\end{proof}

\begin{lem}\label{key1}
Let $\la,\; \mu\in \La^+$. If $\la\geq\mu$ and $g(\la)=g(\mu)$, then
$\la=\mu.$
\end{lem}
\begin{proof}
Since $\la\geq\mu$, we can assume that $\la=\mu+\be$ for some $\be\in Q^+$.
Then
\begin{eqnarray*}
g(\la)&=&g(\mu),\\
t^{\frac{(\la+\rho,\la+\rho)}{2}}&=&t^{\frac{(\mu+\rho,\mu+\rho)}{2}},\\
t^{\frac{(\mu+\be+\rho,\mu+\be+\rho)}{2}}&=&t^{\frac{(\mu+\rho,\mu+\rho)}{2}},\\
t^{(\mu+\be,\be)}&=&1.
\end{eqnarray*}
Hence,
$$
(\mu+\be,\be)=0.
$$
Because of $\mu\in \La^+$, $\be=0$.
\end{proof}

\begin{lem} Let $V^{\q}(\la)\in Ob(\O^{\q}_{int})$. Then the action of
${\Om}^{\q}\cdot{\Xi}^{\q}$ is the scalar
$$g(\la)=t^{\frac{(\la+\rho,\la+\rho)}{2}}.$$
\end{lem}
\begin{proof}
Let $v_{\la}$ be the highest weight vector of $V^{\q}(\la)$. Then
$$
{\Om}^{\q}\cdot{\Xi}^{\q}.v_{\la}=g(\la)v_{\la}.
$$
By Lemma \ref{commute},  we have
${\Om}^{\q}\cdot{\Xi}^{\q}.v=g(\la)v,\quad\forall\  v\in
V^{\q}(\la)$.
\end{proof}

By all above lemmas, similar to Lusztig \cite{Lu} for the
one-parameter ones, we have

\begin{theo} Let $V^{\q}\in Ob(\O^{\q}_{int})$.
Then $V^{\q}$ is completely reducible. \hfill\qed
\end{theo}

\subsection{$R$-matrix}

Let $M,\ M'\in Ob(\O^{\q}_{int})$. The map
$$p_{M,M'}:\,M\otimes M'\lra
M\otimes M'
$$
is defined by
\begin{equation}
p_{M,M'}(m\otimes m')=q_{\mu\nu}^{-1}\,(m\otimes m'),\ \ \forall\;
m\in M_\mu,\; m'\in M_{\nu}'.
\end{equation}
We can take a basis $\{u_k^\be\}_{k=1}^{d_\be}$, ($d_\be =\dim
(U_{\q}^+)_{\be}$) of $(U_{\q}^+)_\be$, and the dual basis
$\{v_k^\be\}_{k=1}^{d_\be}$ of $(U_{\q}^-)_{-\be}$. Then, for any
$x\in (U_{\q}^+)_{\be}$ and $y\in (U_{\q}^-)_{-\be}$,
\begin{equation}
x=\sum_{k=1}^{d_\be}\lan v_k^\be, x\ran_{\q}\,u_k^\be, \qquad
y=\sum_{k=1}^{d_\be}\lan y, u_k^\be\ran_{\q}\,v_k^\be.
\end{equation}

\begin{lem}
Let $x\in (U_{\q}^+)_{\be},y\in (U_{\q}^-)_{-\be}(\be\in Q^+)$. Then
\begin{align}
\De(x)&=\sum_{0\leq\ga\leq\be}\sum_{i,j}\lan
v_i^{\be-\ga}v_j^{\ga},x \ran_{\q}
u_i^{\be-\ga}\om_{\ga}\ot u_j^{\ga},\\
\De(y)&=\sum_{0\leq\ga\leq\be}\sum_{i,j}\lan
y,u_i^{\be-\ga}u_j^{\ga}\ran_{\q} v_j^{\ga} \ot
v_i^{\be-\ga}\om'_{\ga}.
\end{align}
\end{lem}

Denote $\Theta_{\be}=\Theta_{\be}^{-}\ot \Theta_{\be}^{+}$. By a
direct computation, we have the following lemma
\begin{lem}\label{key3} For any $\eta\in Q^+$,
\begin{eqnarray*}
(\De\ot1)\Theta_{\eta}&=&\sum_{0\leq\ga\leq\eta}
(\Theta_{\eta-\ga})_{23}(1\ot \om_{\ga}'\ot 1)(\Theta_{\ga})_{13}\\
&=&\sum_{\be+\ga=\eta}
\Theta_{\be}^{-}\ot\Theta_{\ga}^{-}\om_{\be}'\ot\Theta_{\ga}^{+}\Theta_{\be}^{+},
\end{eqnarray*}
and
\begin{eqnarray*}
(1\ot \De)\Theta_{\eta}&=&\sum_{0\leq\ga\leq\eta}(\Theta_{\eta-\ga})_{12}(1\ot \om_{\ga}\ot 1)
(\Theta_{\ga})_{13}\\
&=&\sum_{\be+\ga=\eta}
\Theta_{\be}^{-}\Theta_{\ga}^{-}\ot\Theta_{\be}^{+}\om_{\ga}\ot\Theta_{\be}^{+}.
\end{eqnarray*}
\end{lem}

Let $M$ and $M' \in Ob(\O^{\q}_{int})$. Define
$$\Theta_{M,M'}^{\q}:\,M\ot M'\lra M\ot M',$$
and $\Theta_\be: \,M_\la\ot M_\mu'\lra M_{\la-\beta}\ot
M_{\mu+\beta}', \ \forall\ \la,\,\mu\in\La$. Note that
$\Theta_{M,M'}^{\q}$ is well-defined.

\begin{theo}\label{mat} Let $M$ and $M' \in Ob(\O^{\q}_{int})$. Then
\begin{equation}
R_{M,M'}^{\q}:=\Theta_{M,M'}^{\q}\circ p_{M',M}\circ P:\, M\otimes M'\lra M'\otimes M
\end{equation}
is an isomorphism of $U_{\q}$-modules, where $P:\,M\otimes M'\lra
M'\otimes M$ is the flip map such that
\begin{equation}
P(m\otimes m')=m'\otimes m,\quad\forall\; m\in M,\,m'\in M'.
\end{equation}
\end{theo}
\begin{proof}It is clear that $R_{M,M'}^{\q}$ is invertible. We shall show that
\begin{equation*}
\Delta(x)R_{M,M'}^{\q}(m\otimes m')=R_{M,M'}^{\q}\Delta(x)(m\otimes m')
\end{equation*}
for any $x\in U_{\q}$, $m\in M_\lambda$ and $m'\in M_\mu'$. In fact,
it suffices to check it for generators $e_i,\,f_i,\,\om_i,\,\om_i' \
(i\in I)$. Here we only check this for $f_i, i\in I$, similarly for
$e_i,\om_i,\,\om_i'$. By Lemma \ref{rm} (iii),
\begin{align*}
\De(f_i) & R_{M,M'}^{\q}(m\ot
m')=q_{\mu\la}^{-1}\De(f_i)\Theta(m'\ot
m)\\
=&\ q_{\mu\la}^{-1}(f_i\otimes
\om_i')\bigl(\sum_{\be\in Q^+}\Theta_{\be-\al_i})(m'\ot m)+q_{\mu,\la}^{-1}(1\ot
f_i)(\sum_{\be\in Q^+}\Theta_\be)(m'\otimes m)\\
=&\ q_{\mu\la}^{-1}(\sum_{\be\in Q^+}\Theta_{\beta-\al_i})(f_i\ot
\om_i)(m'\otimes m)+q_{\mu\la}^{-1}(\sum_{\be\in Q^+}\Theta_\be)(1\ot
f_i)(m'\otimes m)\\
=&\ q_{\mu\la}^{-1}q_{\al_i\la}(\sum_{\be\in
Q^+}\Theta_{\be-\al_i})(f_im'\otimes m)+q_{\mu\la}^{-1}(\sum_{\be\in
Q^+}\Theta_\be)(m'\otimes f_im).
\end{align*}
On the other hand,
\begin{eqnarray*}
R_{M,M'}^{\q}\Delta(f_i)(m\otimes m')&=&R_{M,M'}^{\q}(m\otimes
f_im'+f_im\otimes
\om_i'm')\\
&=&q_{\mu-\al_i\la}^{-1}\Theta(f_im'\otimes
m)+q_{\mu\la-\al_i}^{-1}\Theta(\om_i'm'\otimes f_im)\\
&=&q_{\mu\la}^{-1}q_{\al_i\la}(\sum_{\be\in Q^+}\Theta_{\be-\al_i})(f_im'\otimes m)
+q_{\mu\la}^{-1}(\sum_{\be\in Q^+}\Theta_\be)(m'\otimes f_im).
\end{eqnarray*}

So the proof is complete.
\end{proof}

\begin{coro} For any $M, M', M''\in Ob(\O_{int}^{\q})$, we have the following
quantum Yang-Baxter equation:
\begin{equation*}
R_{12}^{\q}R_{23}^{\q}R_{12}^{\q}=R_{23}^{\q}R_{12}^{\q}R_{23}^{\q}.
\end{equation*}
The category $\O_{int}^{\q}$ is a braided tensor category with the braiding $R_{M,M'}^{\q}$.
\end{coro}

\section{Quantum Shuffle Realization}
\subsection{$\tau$-sesquilinear form on $U_{\q}^+$}

\begin{pro}\label{kashi}
Let $\tau$ be an involution automorphism of $\bK$ such that
$\tau(q_{ij})=q_{ji},\ \forall\; i,$ $j\in I.$ Then there exists a
unique nondegenerate $\tau$-bilinear form $(\,,\,): U_{\q}^{+}\times
U_{\q}^{+}\to \bK$ such that, for any $i\in I$ and $x,y\in
U_{\q}^{+}$,
\begin{align}
(1,1)=1,\quad (xe_i,y)=(x,{\pa_i}y),\quad (e_ix,y)=(x,{_i\pa}y) \label{re1}.
\end{align}
\end{pro}

\begin{proof}
Let $(\,,\,): U_{\q}^{+}\times U_{\q}^{+}\lra \bK$ defined by
$$
(x,y):=\lan\Phi(x),y\ran_{\q},\quad\forall\ x,y\in U_{\q}^{+},
$$
where $\lan\,,\,\ran_{\q}$ is the skew Hopf pairing defined in
Proposition \ref{skew1} and $\Phi$ is the $\tau$-linear automorphism
of $U_{\q}(\mg)$ defined in Lemma \ref{auto}. Since $\Phi$ is
$\tau$-linear, $(\,,\,)$ is $\tau$-sesquilinear. By Lemma \ref{sk}
(iii) and (iv), the condition (\ref{re1}) is satisfied. It is clear
that $(\,,\,)$ is unique and nondegenerate.
\end{proof}
\begin{coro}\label{non1}
Let $x\in U_{\q}^{+}$.  If ${\pa_i}x=0$ for any $ i\in I $, then $x\in\bK$.
\end{coro}

\subsection{Quantum shuffle algebra}
Let $(\F,\cdot)$ be the free associative $\bK$-algebra with $1$ with
generators $w_i\ (i\in I)$. For any $\nu=\sum_i\nu_i \al_i\in Q$, we
denote by $\F_{\nu}$ the $\bK$-subspace of $F$ spanned by the
monomials $w_{i_1}\cdots w_{i_r}$ such that for any $i \in I$, the
number of occurrences of $i$ in the sequence $i_{1},\cdots,i_{r}$ is
equal to $\nu_i$. Then $\F=\oplus_{\nu\in Q}{\F_{\nu}}$ with
$\F_{\nu}$ is a finite dimensional $\bK$-vector space. We have
${\F_{\mu}} {\F_{\nu}}\subset {\F_{\mu+\nu}},\ 1\in {\F_{0}}$ and
$w_i\in {\F_{\al_i}}$. An element $x$ of ${\F}$ is said to be
homogeneous if it belongs to $\F_{\nu}$ for some $\nu$. Let
$|x|=\nu.$ $w[i_{1},\cdots,i_{k}]:=w_{i_1}\cdots w_{i_k}$.

\begin{defi}
The quantum shuffle product $\star$ on $\F$ is defined by
\begin{eqnarray*}
&&1\star x=x\star1=x,\quad\text{for}\ \ x\in \F,\\
&&xw_{i}\star yw_{j} = (xw_{i}\star y)w_{j}+q_{\al_i,\nu+\al_j}
(x\star yw_{j} )w_{i},
\end{eqnarray*}
for $i,j\in I$ and $x\in \F,\ y\in \F_{\nu},\ \mu\in Q^+$.
\end{defi}

\begin{lem}\label{shuffle1} For any $i\neq j \in I$ and $m,l\in \bZ_{+}$, we have
\begin{eqnarray*}
&&w_i^{\star m}\star w_j\star w_i^{\star
l}\\
&=&\sum_{k=0}^{m}\sum_{t=0}^{l}q_{ij}^{k}q_{ji}^{l-t}q_{ii}^{k(l-t)}\bi{m}{k}_{q_{ii}}
\bi{l}{t}_{q_{ii}}(m-k+l-t)_{q_{ii}}!(k+t)_{q_{ii}}! w_i^{m-k+l-t}w_jw_i^{k+t}.
\end{eqnarray*}
\end{lem}
\begin{proof}
See Appendix B.
\end{proof}

\begin{pro}\label{shuffle2} For any $i\neq j \in I$, we have
\begin{equation}
\sum_{k=0}^{1-a_{ij}} (-1)^{k} \binom{1-a_{ij}}{k}_{q_{ii}}
q_{ii}^{\frac{k(k-1)}{2}} q_{ij}^{k}w_{i}^{\star(1-a_{ij}-k)}\star
w_{j}\star w_{i}^{\star k}=0.
\end{equation}
\end{pro}
\begin{proof}
See Appendix C.
\end{proof}

\subsection{Embedding}
We will adopt a similar treatment due to Leclerc \cite{Le} used in
the one-parameter setting. For $w=w[i_{1},\cdots,i_{k}]$, let
${\pa_w}:= {\pa_{i_1}}\cdots {\pa_{i_k}}$ and $
\pa_{w}=\operatorname{Id}$ for $w=1$. Next we introduce a
$\bK$-linear map $\Ga: U_{\q}^{+}\lra (\F,\star)$ defined by
$$
\Ga(x)= \sum_{\substack{{w\in \F}\\{|w|=\mu}}}\pa_{w}(x)w,\quad\forall\ x\in  (U^+_{\q})_{\mu}.
$$
\begin{lem}
$\Ga$ is injective.
\end{lem}
\begin{proof}
Assume $\Ga(x)=0$ for $x\in (U_{\q}^+)_{\mu}$. Then $\pa_{w}(x)=0$
for all $|w|=\mu$. By Corollary \ref{non1}, we have $x=0$, which
implies $\Phi$ is injective.
\end{proof}
Let $ D_i \in{\rm End}(F)\;(i\in I)$ defined as
$$
D_i(1)=0,\quad D_i(w[i_1,\cdots,i_k])=\delta_{i,i_k}w[i_1,\cdots,i_{k-1}].
$$
\begin{lem}\label{skew2}
Each  $D_i\ (i\in I)$ satisfies the relations
\begin{eqnarray*}
&&D_i(w_j)=\delta_{i,j},\\
&&D_i(x\star y)=q_{\al_i\nu}
D_i(x)\star y+x\star D_i(y)
\end{eqnarray*}
for any $y\in \F_{\mu}$ and $x\in \F$.
\end{lem}
\begin{proof}
Let $x=x'w_k, y=y'w_l$. Then
\begin{align*}
D_i(x\star y)=&\ D_i(x'w_k\star y'w_l)\\
=&\  D_i((x'w_k\star y')w_l+q_{\al_k\mu}(x'\star y'w_l)w_k)\\
=&\ \de_{i,l}(x'w_k\star y')+\de_{i,k}q_{\al_k\mu}(x'\star y'w_l)\\
=&\ \de_{i,k}q_{\al_i\mu}(x'\star y'w_l)+(x'w_k\star D_i(y))\\
=&\ q_{\al_i\mu}D_i(x)\star y+x \star D_i(y).
\end{align*}

This completes the proof.
\end{proof}
\begin{theo}\label{iso}
For any $x,y\in U_{\q}^{+}$,  we have $ \Ga(xy)=\Ga(x)\star \Ga(y).
$
\end{theo}
\begin{proof}
By Proposition \ref{shuffle2}, there exists a linear map $\Ga':U_{\q}^{+}\lra
(\F,\star)$ such that
$$
\Ga'(e_i)=w_i,\quad \Ga'(xy)=\Ga'(x)\star\Ga'(y)$$ for $i\in I$ and
$x,\; y\in U_{\q}^+$ . By Lemmas \ref{skew1} and \ref{skew2}, $
\Ga'{\pa_i}=D_i\Ga',\ \forall\ i\in I$. For $x\in U_{\mu}^+,\;
\mu\in Q^+$  and $w=w[i_i,\cdots,i_k]\in \F_{\mu}$, let
${\gamma_w}(x)$ be the coefficient of $w$ in $\Ga'(x).$ Then
$${\gamma_w}(x)=D_{i_1}\cdots D_{i_k}\Ga'(x)=\Ga\, {\pa_{i_1}}\cdots
{\pa_{i_k}}(x)={\pa_w}(x).$$ Hence $\Ga(x)=\Ga'(x).$
\end{proof}

\section{Appendix}

\subsection{Appendix A: The proof of Lemma \ref{primi}}
\begin{equation*}
\begin{split}
\De(u_{ij}^+)&=\sum_{k=0}^{1-a_{ij}}(-1)^{k}
\binom{1-a_{ij}}{k}_{q_{ii}}q_{ii}^{\frac{k(k-1)}{2}}
q_{ij}^{k}\De(e_{i})^{1-a_{ij}-k}\De(e_{j})\De(e_{i})^{k}\\
&=\sum_{k=0}^{1-a_{ij}}\sum_{m=0}^{1-a_{ij}-k}\sum_{n=0}^{k}\binom{1-a_{ij}}{k}_{q_{ii}}
\binom{1-a_{ij}-k}{m}_{q_{ii}}\binom{k}{n}_{q_{ii}}(-1)^{k}q_{ii}^{\frac{k(k-1)}{2}}
q_{ij}^{k}\\
&\quad\times (e_i^{m}\om_i^{1-a_{ij}-k-m}\ot e_i^{1-a_{ij}-k-m})(e_j\ot1 +\om_j\ot e_j)(e_i^{n}\om_i^{k-n}\ot e_i^{k-n})\\
&=\sum_{k=0}^{1-a_{ij}}\sum_{m=0}^{1-a_{ij}-k}\sum_{n=0}^{k}\binom{1-a_{ij}}{k}_{q_{ii}}
\binom{1-a_{ij}-k}{m}_{q_{ii}}\binom{k}{n}_{q_{ii}}\\
&\quad\times (-1)^{k}q_{ii}^{\frac{k(k-1)}{2}}
q_{ij}^{k}q_{ij}^{1-a_{ij}-k-m}q_{ii}^{n(1-a_{ij}-k-m)}(e_i^{m}e_je_i^n\om_i^{1-a_{ij}-m-n}\ot e_i^{1-a_{ij}-m-n})\\
&\quad+\,\sum_{k=0}^{1-a_{ij}}\sum_{m=0}^{1-a_{ij}-k}\sum_{n=0}^{k}\binom{1-a_{ij}}{k}_{q_{ii}}
\binom{1-a_{ij}-k}{m}_{q_{ii}}\binom{k}{n}_{q_{ii}}\\
&\quad\times (-1)^{k}q_{ii}^{\frac{k(k-1)}{2}}
q_{ij}^{k}q_{ji}^{n}q_{ii}^{n(1-a_{ij}-k-m)}(e_i^{m+n}\om_i^{1-a_{ij}-m-n}\om_j\ot e_i^{1-a_{ij}-k-m}e_je_{i}^{k-n})\\
\end{split}
\end{equation*}
\begin{equation*}
\begin{split}&=\sum_{k=0}^{1-a_{ij}}\sum_{m=0}^{1-a_{ij}-k}\sum_{n=0}^{k}\binom{1-a_{ij}}{k}_{q_{ii}}
\binom{1-a_{ij}-k}{m}_{q_{ii}}\binom{k}{n}_{q_{ii}}\\
&\quad\times (-1)^{k}q_{ii}^{\frac{k(k-1)}{2}+n(1-a_{ij}-k-m)}
q_{ij}^{1-a_{ij}-m}(e_i^{m}e_je_i^n\om_i^{1-a_{ij}-m-n}\ot e_i^{1-a_{ij}-m-n})\\
&\quad+\,\sum_{k=0}^{1-a_{ij}}\sum_{m=0}^{1-a_{ij}-k}\sum_{n=0}^{k}\binom{1-a_{ij}}{k}_{q_{ii}}
\binom{1-a_{ij}-k}{m}_{q_{ii}}\binom{k}{n}_{q_{ii}}\\
&\quad\times (-1)^{k}q_{ii}^{\frac{k(k-1)}{2}+n(1-a_{ij}-k-m)}
q_{ij}^{k}q_{ji}^{n}(e_i^{m+n}\om_i^{1-a_{ij}-m-n}\om_j\ot
e_i^{1-a_{ij}-k-m}e_je_{i}^{k-n})\\
&=\sum_{k=0}^{1-a_{ij}}\sum_{m=0}^{1-a_{ij}-k}\sum_{n=0}^{k}\binom{1-a_{ij}-m-n}{k-n}_{q_{ii}}
\binom{1-a_{ij}}{m+n}_{q_{ii}}\binom{m+n}{n}_{q_{ii}}\\
&\quad\times (-1)^{k}q_{ii}^{\frac{k(k-1)}{2}+n(1-a_{ij}-k-m)}
q_{ij}^{1-a_{ij}-m}(e_i^{m}e_je_i^n\om_i^{1-a_{ij}-m-n}\ot
e_i^{1-a_{ij}-m-n})\\
&\quad+\,\sum_{k=0}^{1-a_{ij}}\sum_{m=0}^{1-a_{ij}-k}\sum_{n=0}^{k}\binom{1-a_{ij}-m-n}{k-n}_{q_{ii}}
\binom{1-a_{ij}}{m+n}_{q_{ii}}\binom{m+n}{n}_{q_{ii}}\\
&\quad\times (-1)^{k}q_{ii}^{\frac{k(k-1)}{2}+n(1-a_{ij}-k-m)}
q_{ij}^{k}q_{ji}^{n}(e_i^{m+n}\om_i^{1-a_{ij}-m-n}\om_j\ot
e_i^{1-a_{ij}-k-m}e_je_{i}^{k-n})
\\
&=\sum_{t=0}^{1-a_{ij}}\sum_{u=0}^{1-a_{ij}-t}\sum_{n=0}^{t}\binom{1-a_{ij}-t}{u}_{q_{ii}}
\binom{1-a_{ij}}{t}_{q_{ii}}\binom{t}{n}_{q_{ii}}\\
&\quad\times
(-1)^{u+n}q_{ii}^{\frac{(n+u)(n+u-1)}{2}+n(1-a_{ij}-u-t)}
q_{ij}^{1-a_{ij}-t+n}(e_i^{t-n}e_je_i^n\om_i^{1-a_{ij}-t}\ot e_i^{1-a_{ij}-t})\\
&\quad+\,\sum_{t=0}^{1-a_{ij}}\sum_{u=0}^{1-a_{ij}-t}\sum_{n=0}^{t}\binom{1-a_{ij}-t}{u}_{q_{ii}}
\binom{1-a_{ij}}{t}_{q_{ii}}\binom{t}{n}_{q_{ii}}\\
&\quad\times
(-1)^{u+n}q_{ii}^{\frac{(n+u)(n+u-1)}{2}+n(1-a_{ij}-u-t)}
q_{ij}^{u+n}q_{ji}^{n}(e_i^{t}\om_i^{1-a_{ij}-t}\om_j\ot e_i^{1-a_{ij}-u-t}e_je_{i}^{u})\\
&=\sum_{t=0}^{1-a_{ij}}q_{ij}^{1-a_{ij}-t}\sum_{n=0}^{t}(-1)^nq_{ii}^{\frac{n(n-1)}{2}+n(1-a_{ij}-t)}q_{ij}^{n}
\binom{1-a_{ij}}{t}_{q_{ii}}\binom{t}{n}_{q_{ii}}\\
&\quad\times
\sum_{u=0}^{1-a_{ij}-t}\binom{1-a_{ij}-t}{u}_{q_{ii}}(-1)^{u}q_{ii}^{\frac{u(u-1)}{2}}
(e_i^{t-n}e_je_i^n\om_i^{1-a_{ij}-t}\ot e_i^{1-a_{ij}-t})\\
&\quad+\,\sum_{t=0}^{1-a_{ij}}\sum_{n=0}^{t}(-1)^nq_{ii}^{\frac{n(n-1)}{2}+n(1-t)}
\binom{1-a_{ij}}{t}_{q_{ii}}\binom{t}{n}_{q_{ii}}\\
&\quad\times
\sum_{u=0}^{1-a_{ij}-t}\binom{1-a_{ij}-t}{u}_{q_{ii}}(-1)^{u}q_{ii}^{\frac{u(u-1)}{2}}
q_{ij}^{u}(e_i^{t}\om_i^{1-a_{ij}-t}\om_j\ot e_i^{1-a_{ij}-u-t}e_je_{i}^{u})\\
&=\sum_{t=0}^{1-a_{ij}}q_{ij}^{1-a_{ij}-t}\sum_{n=0}^{t}(-1)^nq_{ii}^{\frac{n(n-1)}{2}+n(1-a_{ij}-t)}q_{ij}^{n}
\binom{1-a_{ij}}{t}_{q_{ii}}\binom{t}{n}_{q_{ii}}\\
&\quad\times \de_{t,1-a_{ij}}
(e_i^{t-n}e_je_i^n\om_i^{1-a_{ij}-t}\ot e_i^{1-a_{ij}-t})\\
\end{split}
\end{equation*}
\begin{equation*}
\begin{split}&\quad+\,\sum_{t=0}^{1-a_{ij}}\sum_{n=0}^{t}(-1)^nq_{ii}^{\frac{n(n-1)}{2}+n(1-t)}
\binom{1-a_{ij}}{t}_{q_{ii}}\binom{t}{n}_{q_{ii}}\\
&\quad\times
\sum_{u=0}^{1-a_{ij}-t}\binom{1-a_{ij}-t}{u}_{q_{ii}}(-1)^{u}q_{ii}^{\frac{u(u-1)}{2}}
q_{ij}^{u}(e_i^{t}\om_i^{1-a_{ij}-t}\om_j\ot e_i^{1-a_{ij}-u-t}e_je_{i}^{u})\\
&=\sum_{n=0}^{1-a_{ij}}(-1)^nq_{ii}^{\frac{n(n-1)}{2}}q_{ij}^{n}
\binom{1-a_{ij}}{n}_{q_{ii}} (e_i^{1-a_{ij}-n}e_je_i^n\ot
1)+\de_{t,0}\sum_{t=0}^{1-a_{ij}}
\binom{1-a_{ij}}{t}_{q_{ii}}\\
&\quad\times\sum_{u=0}^{1-a_{ij}-t}\binom{1-a_{ij}-t}{u}_{q_{ii}}(-1)^{u}q_{ii}^{\frac{u(u-1)}{2}}
q_{ij}^{u}(e_i^{t}\om_i^{1-a_{ij}-t}\om_j\ot e_i^{1-a_{ij}-u-t}e_je_{i}^{u})\\
&=u_{ij}^+\ot
1+\sum_{u=0}^{1-a_{ij}}\binom{1-a_{ij}}{u}_{q_{ii}}(-1)^{u}q_{ii}^{\frac{u(u-1)}{2}}
q_{ij}^{u}(\om_i^{1-a_{ij}}\om_j\ot e_i^{1-a_{ij}-u}e_je_{i}^{u})\\
&=u_{ij}^+\ot 1+\om_i^{1-a_{ij}}\om_j\ot u_{ij}^+.
\end{split}
\end{equation*}

\subsection{Appendix B: The proof of Lemma \ref{shuffle1}}
If $l=0$, we have
$$w_i^{\star m}=(m)_{q_{ii}}!w_i^m.$$
Assume that
Lemma 73 holds for $l$. Then for $l{+}1$, we have
\begin{equation*}
\begin{split}
&w_i^{\star m}\star w_j\star w_i^{\star
l+1}\\
&{=}\sum_{k=0}^{m}q_{ij}^{k}\bi{m}{k}_{q_{ii}}\left\{\sum_{t=0}^{l}q_{ji}^{l{-}t}q_{ii}^{k(l{-}t)}
\bi{l}{t}_{q_{ii}}(m{-}k{+}l{-}t)_{q_{ii}}!(k{+}t)_{q_{ii}}!(w_i^{m{-}k{+}l{-}t}w_jw_i^{k{+}t})\star
w_i\right\}.
\end{split}
\end{equation*}
Then
\begin{equation*}
\begin{split}
&\sum_{t=0}^{l}q_{ji}^{l-t}q_{ii}^{k(l-t)}
\bi{l}{t}_{q_{ii}}(m{-}k{+}l{-}t)_{q_{ii}}!(k{+}t)_{q_{ii}}!
(w_i^{m-k+l-t}w_jw_i^{k+t})\star w_i\\
&=\sum_{t=0}^{l}q_{ji}^{l-t}q_{ii}^{k(l-t)}
\bi{l}{t}_{q_{ii}}(m{-}k{+}l{-}t)_{q_{ii}}!(k{+}t)_{q_{ii}}!\times\\
&\quad\times
\big\{(m{-}k{+}l{-}t{+}1)_{q_{ii}}q_{ji}q_{ii}^{k+t}w_i^{m-k+l-t+1}w_jw_i^{k+t}
+(k{+}t{+}1)_{q_{ii}}w_i^{m-k+l-t}w_jw_i^{k+t+1})\big\}\\
&=\sum_{t=0}^{l}q_{ji}^{l-t+1}q_{ii}^{k(l-t)+k+t}\bi{l}{t}_{q_{ii}}
(m{-}k{+}l{-}t{+}1)_{q_{ii}}!(k{+}t)_{q_{ii}}!w_i^{m-k+l-t+1}w_jw_i^{k+t}\\
&\quad +\, \sum_{t=0}^{l}q_{ji}^{l-t}q_{ii}^{k(l-t)}
\bi{l}{t}_{q_{ii}}(m{-}k{+}l{-}t)_{q_{ii}}!(k{+}t{+}1)_{q_{ii}}!w_i^{m-k+l-t}w_jw_i^{k+t+1}
\\
\end{split}
\end{equation*}
\begin{equation*}
\begin{split}&=\sum_{t=-1}^{l}q_{ji}^{l-t}q_{ii}^{k(l-t)+t+1}\bi{l}
{t{+}1}_{q_{ii}}
(m{-}k{+}l{-}t)_{q_{ii}}!(k{+}t{+}1)_{q_{ii}}!w_i^{m-k+l-t}w_jw_i^{k+t+1}\\
&\quad +\, \sum_{t=-1}^{l}q_{ji}^{l-t}q_{ii}^{k(l-t)}
\bi{l}{t}_{q_{ii}}(m{-}k{+}l{-}t)_{q_{ii}}!(k{+}t{+}1)_{q_{ii}}!w_i^{m-k+l-t}w_jw_i^{k+t+1}\\
&=\sum_{t=-1}^{l}q_{ji}^{l-t}q_{ii}^{k(l-t)} \left\{
q_{ii}^{t+1}\bi{l}
{t{+}1}_{q_{ii}}+ \bi{l}{t}_{q_{ii}}  \right\}(m{-}k{+}l{-}t)_{q_{ii}}!(k{+}t{+}1)_{q_{ii}}!\times\\
&\quad \times w_i^{m-k+l-t}w_jw_i^{k+t+1}\\
&=\sum_{t=-1}^{l}q_{ji}^{l-t}q_{ii}^{k(l-t)} \bi{l{+}1}
{t{+}1}_{q_{ii}}(m{-}k{+}l{-}t)_{q_{ii}}!(k{+}t{+}1)_{q_{ii}}!w_i^{m-k+l-t}w_jw_i^{k+t+1}\\
&=\sum_{t=0}^{l+1}q_{ji}^{l-t+1}q_{ii}^{k(l{-}t{+}1)} \bi{l{+}1}
{t}_{q_{ii}}(m{-}k{+}l{-}t{+}1)_{q_{ii}}!(k{+}t)_{q_{ii}}!w_i^{m-k+l-t+1}w_jw_i^{k+t}.
\end{split}
\end{equation*}

\subsection{Appendix C: The proof of Proposition \ref{shuffle2}}
By Lemma \ref{shuffle1}, we have
\begin{eqnarray*}
&&\sum_{k=0}^{1-a_{ij}} (-1)^{k} \binom{1-a_{ij}}{k}_{q_{ii}}
q_{ii}^{\frac{k(k-1)}{2}} q_{ij}^{k} w_{i}^{\star(1-a_{ij}-k)}\star
w_{j}\star w_{i}^{\star k}\\
&=&\sum_{k=0}^{1-a_{ij}}\sum_{m=0}^{1-a_{ij}-k}\sum_{n=0}^{k}
\binom{1{-}a_{ij}}{k}_{q_{ii}}
\binom{1{-}a_{ij}{-}k}{m}_{q_{ii}}\binom{k}{n}_{q_{ii}}
\\
&&\times(-1)^{k}q_{ii}^{\frac{k(k-1)}{2}+m(k-n)}
q_{ij}^{k+m}q_{ji}^{k-n}(1{-}a_{ij}{-}n{-}m)_{q_{ii}}!(m{+}n)_{q_{ii}}!
 w_i^{1-a_{ij}-n-m}w_jw_i^{m+n}\\
&=&\sum_{k=0}^{1-a_{ij}}\sum_{m=0}^{1-a_{ij}-k}\sum_{n=0}^{k}
\binom{1{-}a_{ij}{-}m{-}n}{k{-}n}_{q_{ii}}
\binom{1{-}a_{ij}}{m{+}n}_{q_{ii}}\binom{m{+}n}{m}_{q_{ii}}
\\
&&\times(-1)^{k}q_{ii}^{\frac{k(k-1)}{2}+m(k-n)} q_{ij}^{k+m}q_{ji}^{k-n}(1{-}a_{ij}{-}n{-}m)_{q_{ii}}!(m{+}n)_{q_{ii}}! w_i^{1-a_{ij}-n-m}w_jw_i^{m+n}\\
&=&\sum_{t=0}^{1-a_{ij}}\sum_{u=0}^{1-a_{ij}-t}\sum_{n=0}^{t}
\binom{1{-}a_{ij}{-}t}{u}_{q_{ii}}
\binom{1{-}a_{ij}}{t}_{q_{ii}}\binom{t}{n}_{q_{ii}}
\\
&&\times(-1)^{n+u}q_{ii}^{\frac{(n+u)(n+u-1)}{2}+(t-n)u} q_{ij}^{u+t}q_{ji}^{u}(1{-}a_{ij}{-}t)_{q_{ii}}!(t)_{q_{ii}}! w_i^{1-a_{ij}-t}w_jw_i^{t}\\
&=&\sum_{t=0}^{1-a_{ij}}\binom{1{-}a_{ij}}{t}_{q_{ii}}q_{ij}^{t}\sum_{n=0}^{t}(-1)^n
q_{ii}^{\frac{n(n-1)}{2}} \binom{t}{n}_{q_{ii}}
\\
&&\times \sum_{u=0}^{1-a_{ij}-t}\binom{1{-}a_{ij}{-}t}{u}_{q_{ii}}(-1)^{u}q_{ii}^{\frac{u(u-1)}{2}}(q_{ii}^{t}q_{ij}q_{ji})^{u}
(1{-}a_{ij}{-}t)_{q_{ii}}!(t)_{q_{ii}}! w_i^{1-a_{ij}-t}w_jw_i^{t}\\
\end{eqnarray*}
\begin{eqnarray*}
&=&\de_{t,0}\sum_{t=0}^{1-a_{ij}}\binom{1-a_{ij}}{t}_{q_{ii}}q_{ij}^{t}
\sum_{u=0}^{1-a_{ij}-t}\binom{1{-}a_{ij}{-}t}{u}_{q_{ii}}(-1)^{u}q_{ii}^{\frac{u(u-1)}{2}}(q_{ii}^{t}q_{ij}q_{ji})^{u}\\
&&\times(1{-}a_{ij}{-}t)_{q_{ii}}!(t)_{q_{ii}}! w_i^{1-a_{ij}-t}w_jw_i^{t}\\
&=&
\sum_{u=0}^{1-a_{ij}}\binom{1{-}a_{ij}}{u}_{q_{ii}}(-1)^{u}q_{ii}^{\frac{u(u-1)}{2}}(q_{ij}q_{ji})^{u}(1-a_{ij})_{q_{ii}}! w_i^{1-a_{ij}}w_j\\
&=&(1-a_{ij})_{q_{ii}}!w_i^{1-a_{ij}}w_j\prod_{n=0}^{-a_{ij}}(1-q_{ii}^{n}q_{ij}q_{ji})\\
&=&0.
\end{eqnarray*}

\bigskip
\bibliographystyle{amsalpha}

\begin{thebibliography}{9999}
\medskip

\bibitem{AE} N. Andruskiewitsch and B. Enriquez, \textit{Examples of compact matrix
pseudogroups arising from the twisting operation}, Comm. Math.
Phys., \textbf{149} (1992), 195--207.

\bibitem{AS1} N. Andruskiewitsch and H.J. Schneider, \textit{Finite quantum groups
and Cartan matrices}, Adv. in Math., \textbf{154} (2000), 1--45.

\bibitem{AS2} N. Andruskiewitsch and H.J. Schneider, \textit{Pointed Hopf
algebras}, New Directions In Hopf Algebra, MSRI publications,
\textbf{43} (2002), 1--68.

\bibitem{AS3} N. Andruskiewitsch and H.J. Schneider, \textit{ A characterization of quantum groups}, J. reine angew.
Math., \textbf{577} (2004), 81--104.

\bibitem{AST} M. Artin, W. Schelter, and J. Tate, \textit{Quantum deformations of
$GL(n)$}, Comm. Pure Appl. Math., \textbf{44} (1991), 879--895.

\bibitem{BH} X. Bai, N. Hu, \textit{Two-parameter quantum groups of exceptional type $E$-series and convex PBW type basis}, arXiv.Math.QA/0605179,  Algebra Colloq., \textbf{15} (4)
(2008), 619--636.

\bibitem{BGH1} N. Bergeron, Y. Gao and N. Hu, \textit{Drinfel'd doubles and
Lusztig's symmetries of two-parameter quantum groups},
  J. Algebra, \textbf{301} (2006), 378--405.

\bibitem{BGH2} N. Bergeron, Y. Gao and N. Hu, \textit{Representations of
two-parameter quantum orthogonal groups and symplectic groups},
AMS/IP, Studies in Advanced Mathematics,  vol. \textbf{39}, pp.
1--21, 2007. arXiv math. QA/0510124.


\bibitem{BKL} G. Benkart, S.~J. Kang and K.~H. Lee, \textit{On the center of two-parameter quantum groups (of type $A$)},
Proc. Roy. Soc. Edingburg Sect. A, \textbf{136} (3), (2006),
445--472.

\bibitem{BW0} G. Benkart and S. Witherspoon, \textit{A Hopf structure for
down-up algebras}, Math. Z., \textbf{238} (3) (2001), 523--553.

\bibitem{BW1} G. Benkart and S. Witherspoon, \textit{Two-parameter quantum
groups (of type $A$) and Drinfel'd doubles}, Algebr. Represent.
Theory, \textbf{7} (2004), 261--286.


\bibitem{BW2} G. Benkart and S. Witherspoon, \textit{Representations of two-parameter quantum
groups (of type $A$) and Schur-Weyl duality}, Hopf Algebras, pp.
65--92, Lecture Notes in Pure and Appl. Math., \textbf{237}, Dekker,
New York, 2004.

\bibitem{BS} S. Burciu, \textit{A class of Drinfeld doubles that are ribbon
algebras}, J. Algebra \textbf{320} (5), (2008), 2053--2078.

\bibitem{CM} W. Chin and I. Musson, \textit{Multi-parameter quantum enveloping algebras}, J. Pure Appl.
Algebra \textbf{107} (1996), 3485--3883.

\bibitem{CV1} M. Costantini, M. Varagnolo, \textit{Quantum double and
multiparameter quantum group}, Comm. in Algebra, \textbf{22} (1994),
6305--6321.

\bibitem{DT} Y. Doi,  M. Takeuchi, \textit{Multiplication alteration by two-cocycles},
Comm. in Algebra,  \textbf{22} (1994), 5715--5732.

\bibitem{FG1}  D. Flores de Chela and J. Green, \textit{
Quantum symmetric algebras},  Algebr. Represent. Theory, \textbf{4}
(2001), 55--76.

\bibitem{F1}C. Fronsdal, \textit{Generalization and exact deformations of quantum groups}, Publ. Res.
Inst. Math. Sci. \textbf{33} (1) (1997), 91--149.

\bibitem{F} C. Fronsdal,  \textit{$q$-Algebras
and arrangements of hyperplanes}, J. Algebra, \textbf{278} (2)
(2004), 433--455.

\bibitem{FG} C. Fronsdal and A. Galindo, \textit{The ideals of free differential algebras}, J. Algebra, \textbf{222}
(1999), 708--746.


\bibitem{Ga}F. Gavarini, \textit{Quantization of Poisson groups}, Pacific
J. Math., \textbf{186} (2) (1998), 217--266.

\bibitem{Gr1}  J. Green, \textit{Hall algebras, hereditary algebras and quantum groups}, Invent. Math., \textbf{120} (1995)
361--377.

\bibitem{Gr2}  J. Green, \textit{Quantum groups, Hall algebras and quantized shuffles}, in Finite
reductive groups (Luminy, 1994), Progr. Math., \textbf{141},
Birkh\"auser, (1997), 273--290.

\bibitem{Gro}P. Gross\'e, \textit{On quantum shuffle and quantum affine algebras},
J. Algebra., \textbf{318} (2) (2007), 495--519.

\bibitem{Ha} T. Hayashi, \textit{Quantum groups and quantum determinants}, J. Algebra, \textbf{152}, (1992), 146--165.

\bibitem{He} I. Heckenberger, \textit{Lusztig isomorphisms for Drinfel'd doubles of bosonizations
of Nichols algebras of diagonal type}, arXiv:0710.4521.

\bibitem{He2} I. Heckenberger, \textit{The Weyl groupoid of a Nichols algebra of diagonal type}, Invent.
Math. \textbf{164} (1), (2006), 175--188.

\bibitem{Ho} T.J. Hodges, \textit{Non-standard quantum groups associated to Belavin-Drinfeld
triples}, Contemp. Math., \textbf{214} (1998), 63--70.

\bibitem{HLT} T.J. Hodges, T. Levasseur, M. Toro, \textit{Algebraic structure
of multi-parameter quantum groups}, Adv. in Math., \textbf{126}
(1997), 52--92.

\bibitem{HK} J. Hong, S. Kang,  \textit{Introduction to Quantum Groups and Crystal Bases},
Graduate Studies in Mathematics, vol. \textbf{42}, Amer. Math. Soc.
Providence, 2002.

\bibitem{HP1} N. Hu, Y. Pei,  \textit{Notes on two-parameter groups (I)}, Sci. in China, Ser. A, \textbf{51} (6) (2008), 1101--1110. math.QA/0702298.

\bibitem{HP2} N. Hu, Y. Pei, M. Rosso, \textit{Notes on two-parameter groups (II)}, preprint.

\bibitem{HRZ} N. Hu, M. Rosso, H. Zhang,  \textit{Two-parameter quantum affine algebra $U_{r,s}(\widehat {\mathfrak {sl}_n})$, Drinfeld realization
and quantum affine Lyndon basis}, Comm. Math. Phys., \textbf{278}
(2) (2008), 453--486.

\bibitem{HS} N. Hu, Q. Shi, \textit{The two-parameter quantum group of exceptional type $G_2$ and
Lusztig's symmetries}, Pacific J. Math., \textbf{230} (2) (2007),
327--346.

\bibitem{HW1} N. Hu and X. Wang, \textit{Convex PBW-type Lyndon bases
and restricted two-parameter quantum groups of type $B$}, Preprint
2006--2008 (submitted).

\bibitem{HW2} N. Hu and X. Wang, \textit{Convex PBW-type Lyndon bases
and restricted two-parameter quantum groups of type $G_2$}, Pacific
J. Math. \textbf{241} (2) (2009) (to appear).

\bibitem{HZ1} N. Hu and H. Zhang, \textit{Vertex representations of two-parameter quantum affine algebras
$U_{r,s}(\widehat{\frak{g}}):$ the simply-laced cases}, Preprint
2006-2007.

\bibitem{HZ2} N. Hu and H. Zhang, \textit{Vertex representations of two-parameter quantum affine algebras
$U_{r,s}(\widehat{\frak{g}}):$ the nonsimply-laced cases}, Preprint
2006-2007.

\bibitem{JC} A. Jacobs, J.F. Cornwell, \textit{Twisting $2$-cocycles for the
construction of new non-standard quantum groups}, J. Math. Phys.,
\textbf{38}, (1997), 5383--5401 .

\bibitem{Ja} J.C. Jantzen, \textit{Lectures on Quantum Groups},
Graduate Studies in Mathematics, vol. \textbf{6}, Amer. Math. Soc.,
Providence, 1996.


\bibitem{Ka}  M. Kashiwara, \textit{On crystal
bases of the $q$-analogue of universal enveloping algebras}, Duke
Math. J., \textbf{63}, (1991), 465--516.

\bibitem{Kh1} V. Kharchenko, \textit{A quantum analog of the Poincar\'e-Birkhoff-Witt theorem}, Algebra and Logic, \textbf{38} (1999), 259--276.

\bibitem{Kh2} V. Kharchenko, \textit{A combinatorial approach to the
quantification of Lie algebras}, Pacific J. Math., \textbf{203}
(2002), 191--233.

\bibitem{KT1} D. Krob, J.-Y. Thibon, \textit{Noncommutative symmetric functions
V: a degenerate version of $U_q(gl_N)$}, Internat. J. Algebra
Comput., \textbf{9} (3-4) (1999),  405--430.

\bibitem{Le} B. Leclerc, \textit{Dual canonical bases, quantum shuffles and $q$-characters},  Math. Z., \textbf{246} (4) (2004), 691--732.

\bibitem{LPWZ} D.-M. Lu, J. H. Palmieri, Q.-S. Wu, and J. J. Zhang,
\textit{Regular algebras of dimension $4$ and their
$A_\infty$-Ext-algebras}, Duke Math. J., \textbf{3}, (2007),
537--584.

\bibitem{Lu} G. Lusztig, \textit{Introduction to Quantum
Groups}, Birkh\"auser Boston, 1993.

\bibitem{Ma} S. Majid,  \textit{Foundations of Quantum Group Theory},
Cambridge U.P., Cambridge, 1995.

\bibitem{P1} Y. Pei,  \textit{Multiparameter quantized
enveloping algebras and their realizations}, Ph.~D. thesis, East
China Normal University, Shanghai, China, 2007.

\bibitem{R1} M. Reineke, \textit{Generic extensions and multiplicative
bases of quantum groups at $q=0$}, Representation Theory,
\textbf{5}, (2001), 147--163.

\bibitem{Re} N. Reshetikhin, \textit{Multiparameter quantum groups and
twisted quasitriangular Hopf algebras}, Lett. Math. Phys.,
\textbf{20}, (1990),  331--335.

\bibitem{Ro0}  M. Rosso,  \textit{Groupes quantiques et alg\'ebres de battage quantiques (Quantum groups and quantum shuffles)},
Comptes Rendus de l'Acad\'emie des Sciences. S\'erie 1,
Mat\'ematique (C. R. Acad. Sci., S\'er. 1, Math.) , \textbf{320},
(1995), 145--148.

\bibitem{Ro1}  M. Rosso,  \textit{Quantum groups and quantum shuffles},
Invent. Math., \textbf{133}  (1998),  399--416.

\bibitem{Ro2}  M. Rosso,  \textit{Lyndon words and universal R-matrices},
Lecture at M.S.R.I. (1999).

\bibitem{Ro3}  M. Rosso,  \textit{Lyndon bases and the multiplicative formula for $R$-matrices},
(2002),  preprint.

\bibitem{Ta} M. Takeuchi, \textit{A two-parameter quantization of $GL(n)$}, Proc. Japan Acad., \textbf{66} (1990),
112--114.


\bibitem{W}
S. Westreich, \textit{Hopf algebras of type $A_n$, twistings and the
FRT-construction}, Algebr. Represent. Theory, {\bf 11}, (2008),
63--82.

\end{thebibliography}

\end{document}